\newif\ifpictures
\picturestrue

\documentclass[12pt]{amsart}

\usepackage{amssymb,amsmath}
\usepackage{hyperref}
\usepackage[dvips]{graphicx}

\ifpictures
\usepackage{psfrag}
\fi

\numberwithin{equation}{section}

\usepackage[latin1]{inputenc}
\usepackage{a4wide}
\usepackage{setspace}
\usepackage{amsopn} 
\usepackage{xspace} 
\usepackage{amsthm}
\usepackage{amssymb}
\usepackage{amsfonts}
\usepackage{stmaryrd}
\usepackage{xcolor}
\usepackage[dvips]{epsfig}
\usepackage{float}
\usepackage{booktabs}

\DeclareMathOperator{\im}{im}

\DeclareMathOperator{\RE}{Re}
\DeclareMathOperator{\IM}{Im}

\DeclareMathOperator{\res}{Res}

\DeclareMathOperator{\conv}{conv }

\DeclareMathOperator{\Log}{Log}

\DeclareMathOperator{\New}{New}

\DeclareMathOperator{\trop}{Trop}
\DeclareMathOperator{\ord}{ord}

\DeclareMathOperator{\eq}{eq}
\DeclareMathOperator{\app}{app}

\newcommand{\eps}{\varepsilon}

\newcommand{\alp}{\alpha}
\newcommand{\lam}{\lambda}

\newcommand{\sig}{\sigma}

\newcommand{\Sig}{\Sigma}

\newcommand{\lf}{\left}
\newcommand{\ri}{\right}
\newcommand{\ra}{\rightarrow}
\newcommand{\Ra}{\Rightarrow}

\newcommand{\Lera}{\Leftrightarrow}

\newcommand{\La}{\Leftarrow}

\newcommand{\w}{\wedge}

\newcommand{\wh}{\widehat}

\newcommand{\ti}{\tilde}
\newcommand{\bs}{\backslash}
\newcommand{\ovl}{\overline}

\newcommand{\lan}{\langle}
\newcommand{\ran}{\rangle}

\newcommand\fC{{\ensuremath{\mathbb{C}}}\xspace}

\newcommand\fN{{\ensuremath{\mathbb{N}}}\xspace}

\newcommand\fQ{{\ensuremath{\mathbb{Q}}}\xspace}
\newcommand\fR{{\ensuremath{\mathbb{R}}}\xspace}

\newcommand\fZ{{\ensuremath{\mathbb{Z}}}\xspace}

\newcommand\cA{{\ensuremath{\mathcal{A}}}\xspace}
\newcommand\cB{{\ensuremath{\mathcal{B}}}\xspace}
\newcommand\cC{{\ensuremath{\mathcal{C}}}\xspace}

\newcommand\cE{{\ensuremath{\mathcal{E}}}\xspace}
\newcommand\cF{{\ensuremath{\mathcal{F}}}\xspace}

\newcommand\cL{{\ensuremath{\mathcal{L}}}\xspace}

\newcommand\cP{{\ensuremath{\mathcal{P}}}\xspace}

\newcommand\cS{{\ensuremath{\mathcal{S}}}\xspace}
\newcommand\cT{{\ensuremath{\mathcal{T}}}\xspace}

\newcommand\cV{{\ensuremath{\mathcal{V}}}\xspace}

\ifpictures
\usepackage{psfrag}
\fi

\numberwithin{equation}{section}

\newtheorem{satz}{Theorem}[section]
\newtheorem{lem}[satz]{Lemma}
\newtheorem{cor}[satz]{Corollary}

\newtheorem{thm}[satz]{Theorem}
\newtheorem{prop}[satz]{Proposition}

\newtheorem*{satz*}{Theorem}
\newtheorem*{lem*}{Lemma}
\newtheorem*{cor*}{Corollary}
\newtheorem*{thm*}{Theorem}
\newtheorem*{prop*}{Proposition}
\newtheorem*{rem*}{Remark}

\theoremstyle{definition}
\newtheorem{defn}[satz]{Definition}
\newtheorem*{defn*}{Definition}
\newtheorem{exa}[satz]{Example}
\newtheorem*{exa*}{Example}

\newcommand{\Fa}[1][\mathbf{w},f]{\cF_{#1}}
\newcommand{\PDe}{\cP_{\Delta}}
\newcommand{\PD}[1][y]{\PDe^{#1}}

\title{Amoebas of genus at most one}
\author{Thorsten Theobald}

\author{Timo de Wolff}

\address{Goethe-Universit\"at, FB 12 -- Institut f\"ur Mathematik,
Postfach 11 19 32, D--60054 Frankfurt am Main, Germany}

\email{\{theobald,wolff\}@math.uni-frankfurt.de}

\thanks{Research supported by DFG grant TH 1333/2-1.}


\begin{document}

\begin{abstract}

The amoeba of a Laurent polynomial $f \in \fC[z_1^{\pm 1}, \ldots, z_n^{\pm 1}]$ is the
image of its zero set $\mathcal{V}(f)$ under the log-absolute-value map. Understanding
the space of amoebas (i.e., the decomposition of
the space of all polynomials, say, with given support or 
Newton polytope, with regard to the existing complement components)
is a widely open problem.

In this paper we investigate the class
of polynomials $f$ whose Newton polytope $\New(f)$ is a simplex and whose 
support $A$ contains exactly one point in
the interior of $\New(f)$. Amoebas of polynomials in this class may have at 
most one bounded complement component.
We provide various results on the space of these amoebas.
In particular, we give upper and lower bounds in terms of the coefficients of 
$f$ for the existence of this complement component and show that the upper 
bound 
becomes sharp under some extremal condition. We establish connections
from our bounds to Purbhoo's lopsidedness 
criterion and to the theory of $A$-discriminants.

Finally, we provide a complete classification of the space of amoebas for the case that the exponent of the inner monomial is the barycenter of the simplex Newton polytope. 
In particular, we show that the set of all polynomials with amoebas of genus 1 is path-connected in the corresponding space of amoebas, which proves a special case of the question on connectivity (for general Newton polytopes) stated by H. Rullg{\aa}rd.

\end{abstract}

\maketitle

\begin{center}
	\textit{Dedicated to Mikael Passare (1959 -- 2011)}
\end{center}
\section{Introduction}
Given a complex Laurent polynomial
$f \in \fC[\mathbf{z}^{\pm 1}] = \fC[z_1^{\pm 1},\ldots,z_n^{\pm 1}]$
the amoeba $\cA(f)$ (introduced by Gel$'$fand, Kapranov, and
Zelevinsky \cite{GelKapZel})
is the image of its variety $\cV(f)$ under the $\log$-absolute-value map
\begin{equation}
\label{eq:amoebadef}
	\Log|\cdot| : \lf(\fC^*\ri)^n \ra \fR^n, \ (z_1,\ldots,z_n) \mapsto (\log|z_1|, \ldots, \log|z_n|) \, ,
\end{equation}
where $\cV(f)$ is considered as a subset of the algebraic torus 
$ \lf(\fC^*\ri)^n = \lf( \fC \setminus \{0\} \ri)^n$. 
Amoebas occur in and have rich connections to various fields of mathematics
(including complex analysis \cite{FoPaTsi1},
the topology of real algebraic curves \cite{Mikh2},
discriminants and hypergeometric 
functions \cite{Nil1,NiPa1}, 
or dynamical systems \cite{EiLiMiWa1}) and in particular form
a cornerstone of tropical geometry (see, e.g., 
\cite{maclagan-sturmfels,Mikh1,PaTsi1}).

By Forsberg, Passare, and Tsikh \cite{FoPaTsi1}, $\cA(f)$ has finitely 
many complement components
whose orders (as introduced in Section~\ref{SubSecAmoebas})
map injectively to the integer points in the 
Newton polytope $\New(f)$ (i.e., the convex hull of the exponents of $f$).
For $\alpha \in \New(f) \cap \fZ^n$ let $E_{\alpha}(f)$ be the (possibly empty) complement
component with order $\alpha$.
Only very little is known concerning the existence and characterization
of the complement components $E_{\alp}(f) \subset \fR^n$  with orders $\alp$
in terms of the coefficients of $f$ (see Section~\ref{SubSecAmoebas}
for some known properties), and thus understanding the
space of amoebas is a widely open field. 
For amoebas of linear polynomials an explicit characterization exists (see \cite{FoPaTsi1}). 
Since in this case there does not exist a bounded complement component
those amoebas are particular instances of amoebas of genus~0. Note that for amoebas of genus 0 all recession cones of complement components can be described explicitly (see \cite[pp. 195-197]{GelKapZel}).

As a step towards better understanding the structure of amoebas of general,
nonlinear varieties, we study a class of polynomials whose amoebas can
have at most one bounded complement component.
For a full-dimensional lattice simplex $\Delta \subset \fR^n$, let
$\PDe$ denote the class of all Laurent polynomials 
with Newton polytope $\Delta$.
Let $\alp(0),\ldots,\alp(n) \in \fZ^n$ be the vertices of an 
$n$-simplex $\Delta$ and $y \in \fZ^n$ be contained in the interior of $\Delta$. Then let
$\PD \subset \PDe$ denote the class of Laurent polynomials of the form
\begin{equation}
	f \ = \ b_0 \cdot \mathbf{z}^{\alp(0)} + b_1 \cdot \mathbf{z}^{\alp(1)} + \cdots + b_n \cdot \mathbf{z}^{\alp(n)} + c \cdot \mathbf{z}^y \, , \quad b_i \in \fC^*, c \in \fC.\label{EquStdPolyn}
\end{equation}
Since $b_0 \in \fC^*$ and $\cV(f) \subset (\fC^*)^n$ we can assume that $\alp(0)$ is the origin and $b_0 = 1$ (otherwise divide $f$ by $b_0 \cdot \mathbf{z}^{\alp(0)}$). Polynomials in $\PD$ have exactly $n+2$ monomials. Note that we do not require that $\# ( \Delta \cap \fZ^n) = n+2$, since the simplex $\Delta$ may contain further lattice points 
as long as the corresponding coefficients are 0. 
For general background on lattice point simplices (with one inner lattice point) see \cite{Averkov, Reznik}), and we remark that $f$ can be regarded as
supported on a circuit (an affinely dependent set whose proper subsets are affinely independent; see, e.g., \cite{bihan-07,prt-2011}).
As explained in Section \ref{SubSecSpine}, $\cA(f)$ can have at most one bounded complement
component and thus there are only two possible homotopy types for $\cA(f)$. 

Our goal is to characterize the  space of the amoebas
of the class of polynomials $\PD$.
After reviewing various properties of amoebas in Section~\ref{SecPreliminaries},
in Section~\ref{SecEquilibPoints} we provide bounds on the coefficients
for the existence and non-existence of the inner complement component.
These bounds -- which are stated in Theorem~\ref{ThmRoughBounds} --
are based on investigating the equilibrium points (as defined
in Definition~\ref{DefEquilibriumPoint}).
We remark that, as a special case, Theorem \ref{ThmRoughBounds} implies 
that maximally sparse polynomials with simplex Newton polytope have 
solid amoebas (Corollary \ref{KorMaximallysparse}); see Nisse \cite{Nisse3}
for a treatment on the solidness of amoebas for more general Newton polytopes.

In Section \ref{SecSharpBounds} we study the points
where (for varying value of $|c|$) the complement component 
appears which provides improved coefficient bounds that even become tight in certain
cases.
Our main results are given in Theorems \ref{ThmImprovedLowerBound} and 
\ref{ThmSharpUpperBound}.

In Section \ref{SecLopsidedness} we connect our results
to Purbhoo's lopsidedness criterion \cite{Purb1} and to the theory of
$A$-discriminants (e.g. \cite{GelKapZel}). Lopsidedness provides a
sufficient criterion for membership to the complement of an amoeba,
and based upon this Purbhoo provided a sequence of approximations which
converge to the amoeba. In our situation we can provide an exact
characterization for genus~1 for all arguments of the inner monomial 
in terms of lopsidedness. See
Theorem~\ref{ThmEquivPerforLopsided}.
With regard to $A$-discriminants we show that a polynomial
$f$ in our class has a complement component of order $y$
such that the upper bound from Theorem \ref{ThmSharpUpperBound} becomes 
sharp if and only if its coefficient vector belongs to the 
$A$-discriminant (Corollary~\ref{co:adisc1}).

In Section \ref{SecBarycenter} we restrict to polynomials in $\PD$ with the additional property that the exponent $y$ of the inner monomial is the barycenter of the simplex spanned by 
$\{\alp(0),\ldots,\alp(n)\}$. For this class we can characterize
the space of amoebas completely
and in particular can show that the set of polynomials
whose amoebas has a complement component of order $y$ is path-connected
(Corollary \ref{KorJackpot}). The question whether the set of polynomials
(w.r.t.\ a fixed support set $A$) having a certain complement component is connected
was marked as an open problem by Rullg{\aa}rd \cite{Rull1}
and is still widely open for non-vertices $\alp(i)$ of $\conv A$.

\section{Preliminaries}
\label{SecPreliminaries}

\subsection{Amoebas}
\label{SubSecAmoebas}

Let $A = \{\alp(1),\ldots,\alp(d)\} \subset \fZ^n$ 
and $f = \sum_{i=1}^d b_i \mathbf{z}^{\alpha(i)} \in \fC[\mathbf{z}^{\pm 1}]$.
The amoeba $\cA(f) \subset \fR^n$ as defined in~\eqref{eq:amoebadef}
is a closed set with non-empty complement 
and each complement component of $\cA(f)$ is convex 
(see \cite{FoPaTsi1,GelKapZel}). The \emph{order map} is given by
$\ord: \fR^n \setminus \cA(f) \to \New(f) \cap \fZ^n$,
\begin{equation}
  \label{eq:ordermap}
  \mathbf{w} \ \mapsto \ \frac{1}{(2\pi i)^n} \int_{\Log|\mathbf{z}| = |\mathbf{w}|} \frac{z_j \partial_j f(\mathbf{z})}{f(\mathbf{z})}
    \frac{dz_1 \cdots dz_n}{z_1 \cdots z_n} \, , \quad 1 \le j \le n \, .
\end{equation}
Since points in the same complement component have the same order,
\eqref{eq:ordermap}
induces an injective map from the set of complement components to
$\New(f) \cap \fZ^n$, and thus the notation
$E_{\alpha}(f)$ for $\alpha \in \New(f) \cap \fZ^n$ (as provided
in the Introduction) is well defined. In particular,
the number of complement components of $\cA(f)$ is bounded by the 
number of lattice points in $\New(f)$.

For the vertices $\alp$ of $\New(f)$, the complement component
$E_{\alp}(f)$ is always non-empty (for every choice of the coefficients
of $f$), while the non-emptiness of $E_{\alpha}(f)$ for non-vertices $\alpha$
depends on the choice of the coefficients of $f$
(see \cite{FoPaTsi1}). For any $\alpha \in A$ it is known 
that there exists some polynomial $f$ supported on $A$ for which
the complement component $E_{\alpha}(f)$ is non-empty \cite{Rull1}.

In order to study the space of amoebas, we can 
identify a polynomial $f = \sum_i b_i \mathbf{z}^{\alpha(i)}$ with its coefficient
vector in $\fC^A$. In our case it is useful and relevant to
consider the subset $\fC^{A}_\Diamond$ of $\fC^A$ with $\New(f) = \conv A$.
Note that for $A := \{\alpha(0), \ldots, \alpha(n),y\}$ with 
$\alpha(0), \ldots, \alpha(n)$ the vertices of an $n$-simplex
and $y$ a lattice point in the interior of $\Delta := \conv A$, 
the space $\fC^{A}_{\Diamond}$ is precisely $\PD$.

For $\alpha \in \New(f) \cap \fZ^n$ let 
$U_{\alp}^A = \{f \in \fC^A_{\Diamond} : E_{\alp}(f) \neq \emptyset\}$
be the set of all polynomials in $\fC^{A}_{\Diamond}$ whose amoeba
has a non-empty complement component of order $\alp$.
Note that the map $\fC^A_{\Diamond} \ra \fN, f \mapsto \# \{E_{\alp}(f) \neq \emptyset\}$ is lower semicontinuous and thus the sets $U_{\alpha}^A$ are open sets (see \cite[Prop. 1.2]{FoPaTsi1}, \cite{Rull1}). Furthermore, 
all $U_{\alp}^A$ are non-empty and semialgebraic sets (see \cite{Rull1}).

If one considers the image of a variety under the argument map, rather than
the $\Log|\cdot|$-map, the resulting set is called \emph{coamoeba} and has
recently also attracted attention (see~\cite{NiPa1,
Nisse3,nisse-sottile}).

\subsection{The tropicalization, the spine, and the complement-induced tropicalization}
\label{SubSecSpine}

We introduce four polyhedral complexes which are naturally associated with an amoeba:
the \textit{tropical hypersurface}, the \textit{equilibrium}, the \textit{complement-induced tropical hypersurface} 
and the \textit{spine}.

Recall that the \textit{tropical semiring} $(\fR \cup \{-\infty\},\oplus,\odot)$ is given by the operations $a \oplus b := \max(a,b)$ and  $a \odot b:= a + b$ (where some expositions prefer the minimum instead of the maximum). For a tropical polynomial $h$, the \textit{tropical hypersurface} $\cT(h)$  is the set of points where the maximum is attained at least twice (see, e.g., \cite{Gath1, RGStuTh1}). It is well known that tropical hypersurfaces are polyhedral complexes which are geometrically dual to a subdivision of the Newton polytope of $h$.

Let $f = \sum_{i = 1}^d m_i(\mathbf{z}) = \sum_{i = 1}^m b_i \mathbf{z}^{\alp(i)}$ with terms $m_i$ and coefficients $b_i \in \fC$, and 
$C := \{\alp \in \New(f) \cap \fZ^n \, : \, E_{\alp}(f) \neq \emptyset\}$ 
be the set of orders of the existing complement components. The \textit{tropicalization} of $f$ is the tropical polynomial (say, in the variables $\mathbf{w}$)
\begin{eqnarray*}
		\trop(f)	& =	& \bigoplus_{i = 1}^d \, \log|b_i| \odot \mathbf{w}^{\alp(i)},
	\end{eqnarray*}
and the \textit{complement-induced tropicalization} is
	\begin{eqnarray*}
		\trop(f_{|C})	& =	& \bigoplus_{\alp{(i)} \in C} \log|b_i| \odot \mathbf{w}^{\alp(i)}
	\end{eqnarray*}
(see, e.g., \cite{PR1,PaTsi1,Rull1}).
We set $\cC(f) = \cT(\trop(f_{|C}))$.

Define the \textit{(norm-induced) equilibrium} $\cE(f)$ of $f$ as the following superset of $\cT(\trop(f))$
(and of $\cT(\trop(f_{|C})$),
\begin{equation}
\label{eq:modulareq}
\cE(f) \, = \, \lf\{\mathbf{w} \in \fR^n \, : \,   
  | m_i(\Log^{-1}|\mathbf{w}|) | = | m_j(\Log^{-1}|\mathbf{w}|) | 
  \text{ for some } 1 \le i \neq j \le d \ri\}.
\end{equation}

The Ronkin function
\begin{eqnarray*}
N_f : \fR^n \ra \fR, \qquad \mathbf{w} \ \mapsto \ \frac{1}{(2 \pi i)^n} \int_{\Log^{-1}|\mathbf{w}|} \frac{\log|f(z_1,\ldots,z_n)|}{z_1 \cdots z_n} \ dz_1 \cdots dz_n
\end{eqnarray*}
of a polynomial $f$ is a convex function which is affine linear on the complement components of $\cA(f)$ and can be interpreted as the average value of the fiber $\Log^{-1}|\mathbf{w}|$ (\cite{Ron1}, 
cf.\ \cite{PR1}).
The gradient of $N_f(\mathbf{w})$ for a $\mathbf{w} \in \fR^n \bs \cA(f)$ coincides with the order of the corresponding complement component (\cite{FoPaTsi1}).

By the affine linearity of $N_f(\mathbf{w})$ 
on every $E_{\alp}(f)$, we have for all $\mathbf{w} \in E_{\alp}(f)$ 
that $N_f(\mathbf{w}) = \beta_{\alp} + \lf\langle \alp,\mathbf{w} \ri\rangle$
with Ronkin coefficient
\begin{eqnarray}
	\beta_{\alp}	& =	& \log|b_i| + \RE \lf[\frac{1}{(2\pi i)^n} \int_{\Log^{-1}|0|} \log\lf(\frac{f(\mathbf{z})}{b_i \cdot \mathbf{z}^{\alp}}\ri) \frac{d z_1 \w \ldots \w d z_n}{z_1 \cdots z_n}\ri].\label{EquRonkinCoeff}
\end{eqnarray}
The \textit{spine} $\cS(f)$ of $\cA(f)$ 
is the tropical hypersurface of the tropical polynomial 
$\bigoplus_{\alp \in C} \beta_{\alp} \odot \mathbf{w}^{\alp}$
and is therefore dual to an integral, regular subdivision of $\New(f)$ (cf. \cite{PR1,Rull1}).
 
The spine $\cS(f)$ is a strong deformation retract of the amoeba
$\cA(f)$ (see \cite{PR1}). In general, the complement-induced
tropical hypersurface $\cC(f)$ 
is not a deformation retract of $\cA(f)$. 
However, for a certain rich subclass of Laurent polynomials 
we have (\cite[Theorem 8, p. 33 and the proof of Theorem 12, p. 36]{Rull1}):

\begin{lem}[Rullg{\aa}rd]
Let $f \in \fC[\mathbf{z}^{\pm 1}]$ with at most $2n$ monomials such that for all $k \in \{1,\ldots,n-1\}$ no $k+2$ of its exponent vectors lie in an affine 
$k$-dimensional subspace. Then $\cC(f)$ is a strong deformation 
retract of $\cA(f)$.
\label{LemRullgard}
\end{lem}

This implies in particular that for all polynomials $f$ in $\PD$ the 
complement-induced tropical hyperplane $\cC(f)$ is a deformation retract of their amoeba $\cA(f)$. Thus there are just two possible homotopy types for polynomials $f$ in $\PD$ since the tropical hypersurface 
$\cC(f)$ is dual to a regular subdivision of the point set $A$ which has, since it is a circuit, only two possible triangulations (see \cite[Chapter 7, p. 217]{GelKapZel}).

Although the spine (or in case of $\PD$ even $\cC(f)$) is a
tropical hypersurface, it is nevertheless difficult to compute 
the homotopy of $\cA(f)$. Both the definitions of $\cS(f)$ and $\cC(f)$ 
depend on $C$ and in general do not depend continuously 
on the coefficients of $f$ (\cite{PR1}). 

\begin{exa} Given $\Delta = \conv \{0,(2,1),(1,2)\}$, $y=(1,1)$, and
$f = 1 + z_1^2z_2 + z_1z_2^2 - 4z_1z_2$ in $\PD$, 
Figure~\ref{fi:bsp1} depicts
 $\cA(f)$, $\cS(f)$, $\cC(f)$ and $\cE(f)$.
  \ifpictures
\begin{figure}[t]
	\begin{center}
		\includegraphics[width=0.48\linewidth]{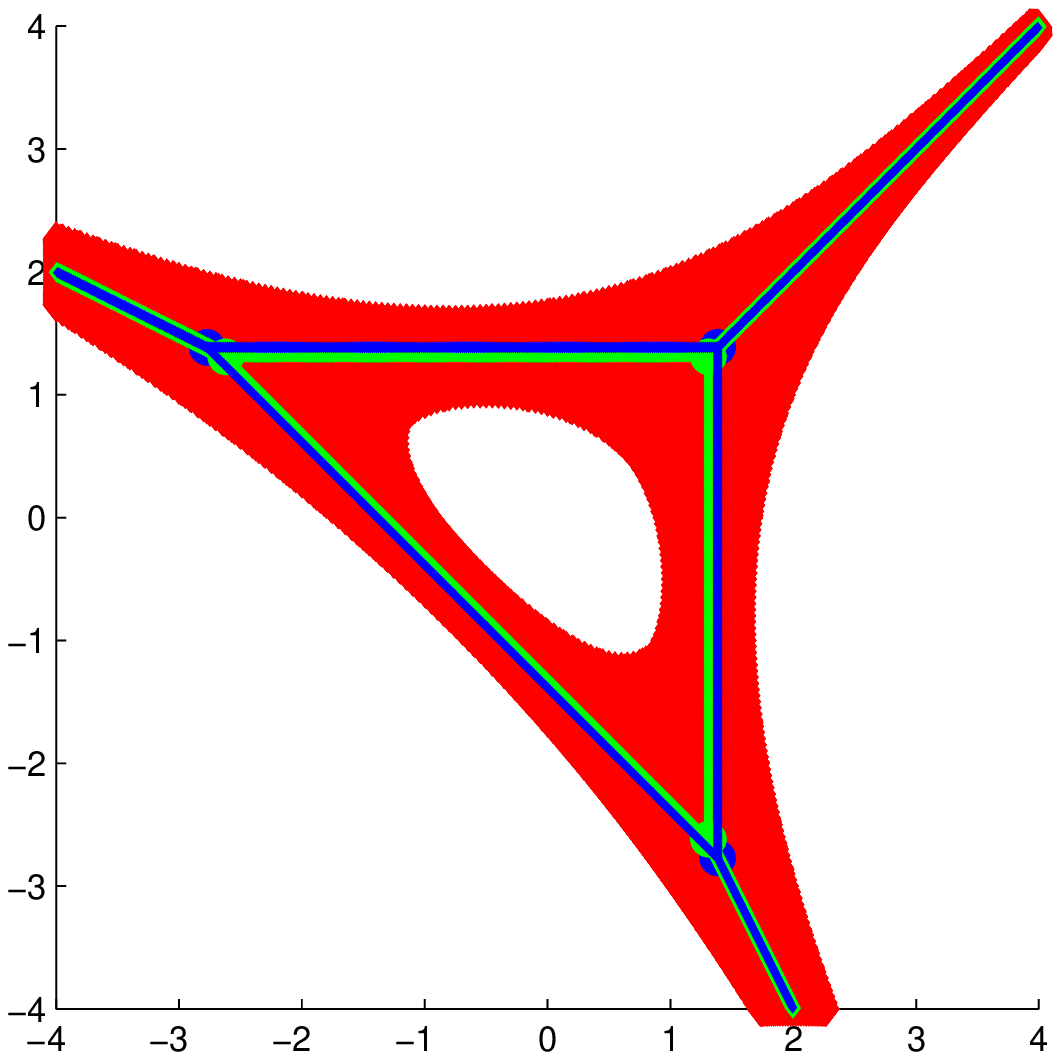}
		\includegraphics[width=0.48\linewidth]{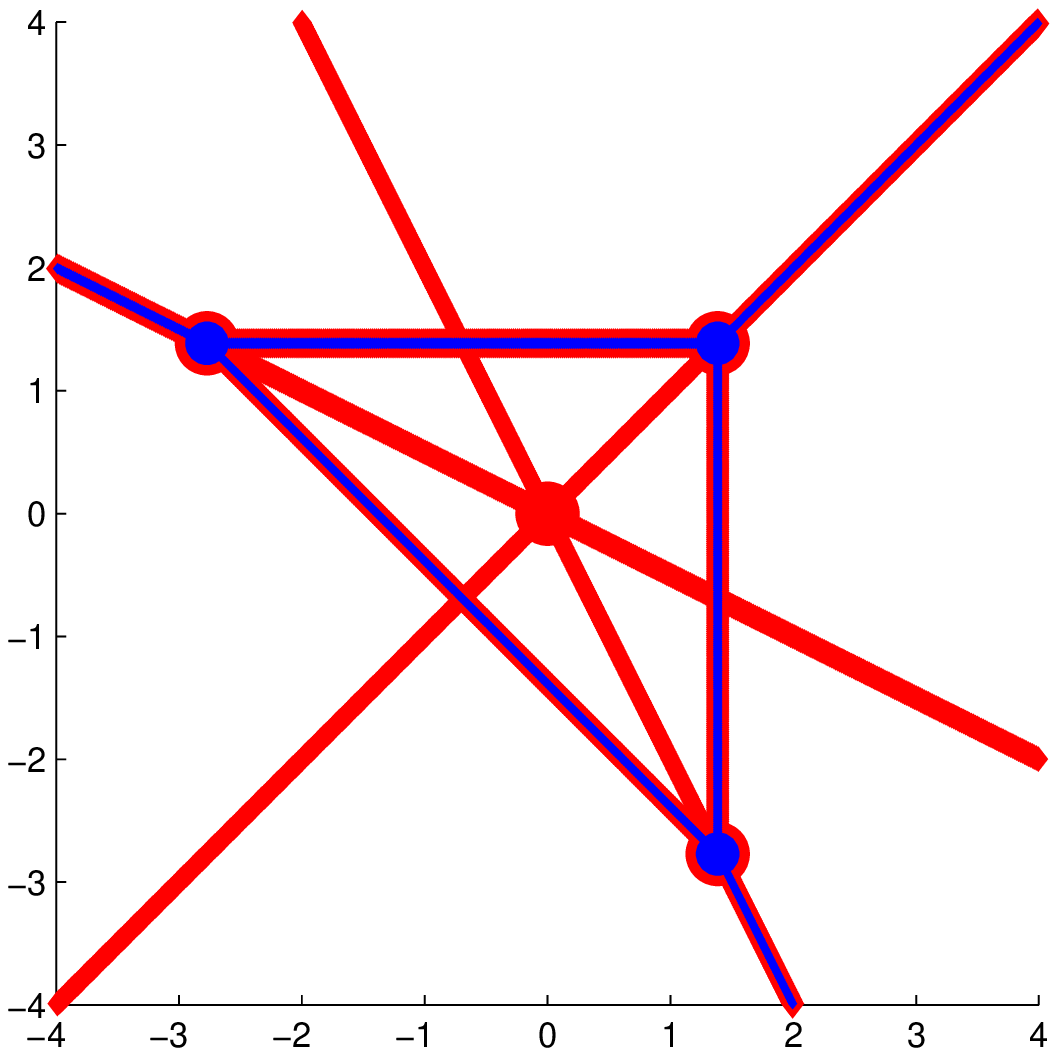}
	\caption{Let $f = 1 + z_1^2z_2 + z_1z_2^2 - 4z_1z_2$. Left picture: the amoeba $\cA(f)$ (red) with the spine $\cS(f)$ (green, light) and the complement-induced tropical hypersurface $\cC(f)$ (blue, dark). Note that on the outer tentacles $\cS(f)$
and $\cC(f)$ coincide.
Right picture: the equilibrium $\cE(f)$ (red) together with $\cC(f)$ (blue, dark). Note that $\cC(f) \subset \cE(f)$. The equilibrium points introduced in Definition \ref{DefEquilibriumPoint} are marked by big red points.}
	\label{fi:bsp1}
	\end{center}
\end{figure}
\fi
\end{exa}

\subsection{Fibers}
\label{SubsecFibers}
Let $f \in \fC\lf[\mathbf{z}^{\pm 1}\ri]$. For our investigations 
the fibers of certain points $\mathbf{w} \in \Log\lf|\lf(\fC^*\ri)^n\ri|$ under the $\Log|\cdot|$-map play a key role. Any such fiber is a real $n$-torus $[0,2\pi)^n$, and $f$ induces a function $\Fa$ 
on the fiber of a $\mathbf{w} \in \fR^n$:
\begin{eqnarray}
	\Fa : [0,2\pi)^n \ra \fC, \quad \phi \mapsto f\lf(e^{w_1} \cdot e^{i \cdot \phi_1},\ldots,e^{w_n} \cdot e^{i \cdot \phi_n}\ri).
\label{EquFiber}
\end{eqnarray}

Notice that a point $\mathbf{w}$ is contained in $\cA(f)$ if and only if there exists some $\phi \in [0,2\pi)^n$ with $\Fa(\phi) = 0$. 

\medskip

\section{Equilibrium points and bounds for the inner complement component}
\label{SecEquilibPoints}
From now on, we study polynomials $f \in \PD$ of the form (\ref{EquStdPolyn}). 
The monomials $b_i \mathbf{z}^{\alp(i)}$ are called 
the \emph{outer monomials} 
and $c \mathbf{z}^y$ is called the \emph{inner monomial}.
These polynomials form a ``simplest'' class
of the polynomials where the characterization of the amoeba becomes ``difficult''.
Since an exact description of the complement components (and, in particular, the homotopy) is not available, one of our main goals is to provide bounds on the coefficients to determine the homotopy type of $\cA(f)$.  In this section, we focus on bounds which are obtained by investigating \emph{equilibrium points} (as introduced in Definition~\ref{DefEquilibriumPoint}).

As a starting point, recall that the complement components of amoebas of linear polynomials are well understood. By Forsberg, Passare and Tsikh \cite[Proposition 4.2]{FoPaTsi1}, for a linear polynomial $f := b_0 + \sum_{i = 1}^{n} b_i  z_i$ and a point $\mathbf{z} \in (\fC^*)^n$, $\Log|\mathbf{z}| \in \fR^n \setminus \cA(f)$ if and only if $|b_0| > \sum_{j = 1}^n \lf|b_j  z_j \ri|$ or $\lf|b_i  z_i \ri| > |b_0| + \sum_{j \neq i} \lf|b_j  z_j \ri|$ for some $i \in \{1,\ldots,n\}$. The following statement captures a slight generalization of this result to Newton polytopes that might contain interior lattice points.

\begin{thm}
Let $f := \sum_{i = 0}^n b_i \mathbf{z}^{\alp(i)}$ such that the convex hull of $\{\alp(0),\ldots,\alp(n)\}$ is an $n$-simplex. For $\mathbf{z} \in (\fC^*)^n$ we have $\Log|\mathbf{z}| \in \fR^n \setminus \cA(f)$ if and only if $\lf|b_i  \mathbf{z}^{\alp(i)}\ri| > \sum_{j \neq i} \lf|b_j  \mathbf{z}^{\alp(j)}\ri|$ for some $i \in \{0,\ldots,n\}$.
\label{ThmMaximallSparseCase}
\end{thm}

The $f$ from Theorem \ref{ThmMaximallSparseCase} are a particular kind of \emph{maximally sparse polynomial}, 
where an arbitrary polynomial $f$ is called \emph{maximally sparse} if 
for all non-vertices $\alpha$ of $\New(f)$ we have $b_i = 0$.

For the convenience of the reader we provide a proof of Theorem \ref{ThmMaximallSparseCase} which is analogous to the proof of statement \cite[Proposition 4.2]{FoPaTsi1}.

\begin{proof}
The direction ``$\La$'' is obvious. For the converse direction let $\mathbf{z} \in \left(\fC^*\right)^n$ with $|b_i \mathbf{z}^{\alp(i)}| \leq \sum_{j \neq i} |b_j \mathbf{z}^{\alp(j)}|$ for all $i \in \{1,\ldots,n\}$. Since the case $n = 1$ is trivial, assume $n \geq 2$. We normalize such that $\alp(0) = 0 \in \fZ^n$ and $\arg(b_0) = 0 \in [0,2\pi)$. 

Order the monomials by norm so that $|b_j  \mathbf{z}^{\alp(j)}| \leq |b_{j+1}  \mathbf{z}^{\alp(j+1)}|$ for $j \in \{0,\ldots,n\}$ and let $m$ denote the largest integer such that $\sum_{j = 0}^{m-1} |b_j  \mathbf{z}^{\alp(j)}| < \sum_{j = m}^n|b_j  \mathbf{z}^{\alp(j)}|$. By choice of $\mathbf{z}$ we have $m < n$. We denote $t_1 := \sum_{j = 0}^{m-1} |b_j  \mathbf{z}^{\alp(j)}|$, $t_2 := |b_m  \mathbf{z}^{\alp(m)}|$ and $t_3 := \sum_{j = m+1}^n |b_j  \mathbf{z}^{\alp(j)}|$. By the choice of $m$ we have $t_1 + t_2 \geq t_3$, $t_1 + t_3 \geq t_2$ and $t_2 + t_3 \geq t_1$. Hence, $t_1,t_2,t_3$ form the edge lengths of a triangle 
and thus there are $\psi_1,\psi_2 \in [0,2\pi)$ with 
\begin{eqnarray*}
	\sum_{j = 0}^{m-1} |b_j  \mathbf{z}^{\alp(j)}| + |b_m  \mathbf{z}^{\alp(m)}| \cdot  e^{i \cdot \psi_1} + \sum_{j = m+1}^{n} |b_j  \mathbf{z}^{\alp(j)}| \cdot e^{i \cdot \psi_2} & = & 0.
\end{eqnarray*}
Since the integer vectors
$\alp(1), \ldots, \alp(n)$ are linearly independent, we can
find $\phi \in [0,2 \pi)^n$ such that 
$\sum_{j = 0}^{n} b_j  |\mathbf{z}|^{\alp(j)} \cdot
e^{i \cdot \lf\langle \alp(i),\phi \ri\rangle} = 0$
and thus $\Log|\mathbf{z}| \in \cA(f)$.

Finally, one can show that all extreme points of the closure of $\cA(f)$ satisfy the required inequalities which we omit here.
\end{proof}

Thus, the class $\PD$ is a natural generalization of maximally sparse polynomials with simplex Newton polytope.
Note that the above proof technique does not extend to 
the case of supports with interior integer points since then
the set of all exponent vectors is not affinely independent.

In the following we often write $f$ as a sum of monomials:
\begin{eqnarray}
	f(\mathbf{z})	& =	& m_0(\mathbf{z}) + m_1(\mathbf{z}) + \cdots + m_n(\mathbf{z}) + m_y(\mathbf{z}) \label{EquStdPolynMonomNotation}
\end{eqnarray}
with each $m_i(\mathbf{z})$ representing the corresponding monomial of $f$ 
in the notation of (\ref{EquStdPolyn}).
By our remarks after Lemma~\ref{LemRullgard} 
there are only two possible homotopy types
for the amoeba of a polynomial $f \in \PD$, and it is useful to
introduce the following equilibrium points related to the 
equilibrium $\cE(f)$ from~\eqref{eq:modulareq}.

\begin{defn}
For $f \in \PD$ of the form~\eqref{EquStdPolynMonomNotation},
let $\eq(y)$ be the point of the equilibrium $\cE(f)$ 
where at least all monomials but $m_y$ have the same norm, i.e., $|m_0(\eq(y))| = \cdots = |m_n(\eq(y))|$.
Similarly, for $0 \le j \le n$ let $\eq(j)$
be the point in $\cE(f)$ where at least all monomials 
but $m_j$ have the same norm. We call $\eq(y), \eq(0), \ldots, 
\eq(n)$, the \emph{(norm-induced) equilibrium points}.
\label{DefEquilibriumPoint}	
\end{defn}

Let $M \in \fZ^{n \times n}$ be the matrix with columns $\alp(1), \ldots, \alp(n)$.

\begin{lem} If $\alpha(0) = 0$ and $b_0 = 1$ then
the equilibrium point $\eq(y) \in \fR^n$ is the unique solution $\mathbf{x} \in \fR^n$ of the system of linear equations $M^t \cdot \mathbf{x} = - \Log|(b_1,\ldots,b_n)^t|$.
\label{LemEquilibPointCoord}
\end{lem}

\begin{proof}
The point $\eq(y)$ is the point where all monomials $m_0(\mathbf{z}),\ldots,m_n(\mathbf{z})$ are in equilibrium. Hence $\eq(y)$ satisfies the $n$ linear equations
\begin{eqnarray*}
	\log |b_i| + \lf\langle \mathbf{w}, \alp{(i)} \ri\rangle	& =	& \log|b_0| + \lf\langle \mathbf{w}, \alp{(0)} \ri\rangle.
\end{eqnarray*}
Since $\alp(0) = 0$, each of these coincides with one row of the linear system $M^t \cdot \mathbf{x} = - \Log|b|$.
\end{proof}

The following lemma states how the spine $\cS(f)$ of the amoeba $\cA(f)$ is related to $\cC(f)$.

\begin{lem}
Let $f \in \PD$.
\begin{enumerate}
	\item[(a)] If $\cA(f)$ is solid then the inner vertex of $\cS(f)$ is the equilibrium point $\eq(y)$ and $\cS(f)$ coincides with the complement-induced tropicalization $\cC(f)$.
	\item[(b)] If $\cA(f)$ has genus 1 then $\cS(f)$ and $\cC(f)$ are homotopy equivalent, their inner simplices $\Sig_{\cS(f)}$ and $\Sig_{\cC(f)}$ are similar and all faces not belonging to the inner simplices coincide in all points lying outside of both inner simplices.
\end{enumerate}
\label{LemRelSpinEquilibPerforated}
\end{lem}

\begin{proof}
(a) If $\cA(f)$ is solid then the order of any complement component of $\cA(f)$ is a vertex of $\New(f)$ and hence for every Ronkin coefficient $\beta_{\alp{(i)}}$ we have $\beta_{\alp{(i)}} = \log|b_i|$ and therefore $\cS(f) = \cC(f)$.

(b) Let $\cA(f)$ have genus 1.
Since $n = 1$ is trivial, we can assume $n \geq 2$.
$\cS(f)$ and $\cC(f)$ coincide in all points lying outside of both inner simplices since for any vertex $\alp(i)$ of $\New(f)$ we have $\beta_{\alp(i)} = \log|b_i|$. As $n \ge 2$,
homotopy equivalence follows from Lemma \ref{LemRullgard}. Since $\cS(f)$ and $\cC(f)$ 
are tropical hypersurfaces dual to the same triangulation of $\New(f)$, $\Sig_{\cS(f)}$ and $\Sig_{\cC(f)}$ are similar.
\end{proof}

\begin{lem}
Let $n \geq 2$, $\alpha(0) = 0$, $b_0 = 1$ and $f \in \PD$ such that $\cA(f)$ has genus 1.
\begin{enumerate}
	\item[(a)] If $\mathbf{z} \in \lf(\fC^*\ri)^n$ with $\Log|\mathbf{z}| = \eq(y)$ then $|m_y(\mathbf{z})| > 1$.
	\item[(b)] The equilibrium point $\eq(y)$ is contained in the interior of the simplex with vertices $\eq(0),\ldots,\eq(n)$.
\end{enumerate}
\label{LemInnerEquiPoint}
\end{lem}

\begin{proof}
(a) Assume that $|m_y(\mathbf{z})| \leq 1$. Due to definition of $\eq(y)$ and $\Log|\mathbf{z}| = \eq(y)$ we know $|m_i(\mathbf{z})| = 1$ for all $i \in \{0,\ldots,n\}$. Hence, we have $\eq(y) \in \cC(f)$. By Lemma \ref{LemEquilibPointCoord} $\eq(y)$ is the unique point where the infinite cells of $\cC(f)$ intersect. Thus, $\cC(f)$ has genus 0. This yields a contradiction since $\cA(f)$ has genus 1 and $\cC(f)$ is a deformation retract of $\cA(f)$ for $n \geq 2$ by Lemma \ref{LemRullgard}.

(b) Let $\Sig'$ be the simplex with vertices $\eq(0),\ldots,\eq(n)$. By definition of $\cC(f)$ we have for all $\mathbf{z} \in \lf(\fC^*\ri)^n$: If $|m_y(\mathbf{z})| > |m_i(\mathbf{z})|$ for all $i \in \{0,\ldots,n\}$, then $\Log|\mathbf{z}|$ is contained in the interior of $\Sig'$. With (a) the assertion follows.
\end{proof}

Let $f \in \PD$, and consider $f$ with a varying $\arg(c)$. 
An angle $\arg(c)$ is called in
\emph{extreme opposition} if there exists some $\mathbf{z} \in (\fC^*)^n$ with
\begin{equation} \label{Equ:extreme}
  \arg(m_y(\mathbf{z})) \ = \ \arg(m_i(\mathbf{z})) + \pi \pmod{2 \pi} \, , \quad 0 \le i \le n \, .
\end{equation}
Since condition~\eqref{Equ:extreme} is actually independent of the 
norm of $\mathbf{z}$ (and also of the norm of the coefficients), 
we call $\arg(\mathbf{z})$ an \emph{extremal phase}.

\begin{lem}
Let $f$ be in $\PD$, where we consider $\arg(c)$ as parameter.
Then there always exists some choice of $\arg(c)$ such that $\arg(c)$ is in extreme opposition.
\end{lem}

\begin{proof} By multiplying $f$ with a Laurent monomial, we can assume
$\alp(0)=0$ and $b_0 = 1$.

Setting $\phi := \arg(\mathbf{z})$, the condition~\eqref{Equ:extreme} is 
a linear condition in $\phi$. Using the non-singular 
integral matrix $M$ introduced above, the image of $[0,2\pi)^n$
under the mapping $\phi \mapsto M \phi$ is a $D$-fold covering of
$[0,2 \pi)^n$ where $D := \det(M)$. Hence, there exists
$\phi \in [0,2\pi)^n$ with
\begin{eqnarray*}
  M^t \cdot \phi	& =	& - (\arg(b_1),\ldots,\arg(b_n))^t \mod 2 \pi,
\end{eqnarray*}
and indeed the number of distinct solutions for $\phi$
in $[0,2\pi)^n$ is $D$.
Setting $\arg(c) := \pi - \langle \phi,y \rangle$ we obtain the result.
\end{proof}

In order to study the amoebas of polynomials in $\PD$, we investigate the parametric family of polynomials
\begin{eqnarray}
	f_\kappa	& :=	& \left[|c| \cdot e^{i \cdot \arg(c)} \cdot \mathbf{z}^y + \sum_{i = 0}^n b_i \cdot \mathbf{z}^{\alp(i)}\right]_{|c| = \kappa} \ = \ \kappa \cdot e^{i \cdot \arg(c)} \cdot \mathbf{z}^y + \sum_{i = 0}^n b_i \cdot \mathbf{z}^{\alp(i)}
	\label{Equfkappa}	
\end{eqnarray}
 in  $\PD$. Recall that, for a fixed $\kappa_1 \in \fR_{> 0}$, $E_y(f_{\kappa_1}) \subset \fR^n$ denotes the set of all points belonging to the complement of $\cA(f_{\kappa_1})$ which have the order $y$.

For a parametric family $f_\kappa$ we are interested in those parameters $\kappa$ where the genus of $\cA(f_\kappa)$ changes. We say that $\cA(f_\kappa)$ \textit{switches} from genus 0 to 1 at $\kappa_0$, if $E_y(f_{\kappa_0}) = \emptyset$ and for every (sufficiently small) $\eps > 0$ we have $E_y(f_{\kappa_0 + \eps}) \neq \emptyset$. Note that, 
for sufficiently large $\kappa$, $\cA(f_\kappa)$ is always of genus 1 (e.g. by the lopsidedness criterion; see Section \ref{SecLopsidedness}).

For a parameter value $\kappa_1 \in \fR$ with $E_y(f_{\kappa_1}) \neq \emptyset$ we
are furthermore interested in characterizing the point where the complement component $E_y$
appears first (with respect to values $\kappa < \kappa_1$ in the parametric
family).
Formally, we say that the inner complement component $E_y(f_{\kappa_1})$ \textit{appears first} 
at $\mathbf{w} \in \Log\lf|\lf(\fC^*\ri)^n\ri|$ if the following conditions hold:
\begin{enumerate}
	\item[(a)] $\mathbf{w} \in E_y(f_{\kappa_1})$, and
	\item [(b)] there exists a $\kappa_0 < \kappa_1$ such that $E_y(f_{\kappa_0}) = \emptyset$ and for every $\kappa \in [\kappa_0,\kappa_1]$ we have $E_y(f_{\kappa}) = \emptyset$ or $\mathbf{w} \in E_y(f_{\kappa})$.
\end{enumerate}
For every such $\kappa_1$ this point is unique and will be denoted by 
$\app(f_{\kappa_1})$.

Let $K \subset \fR_{\geq 0}$ for some given parametric family $f_{\kappa}$ denote parameters where $\cA(f_\kappa)$ switches from genus 0 to 1. Then we say $f_\kappa$ \textit{switches the last time} from genus 0 to 1 at $\kappa^* := \max K$. In the following we are in particular interested in the corresponding point $\app(f_{\kappa^*})$ where the inner complement component \textit{finally appears} and which we denote as $\mathbf{a}(f_{\kappa})$.

Let $M_j$ be the matrix obtained by replacing the $j$-th column of $M$ by $y$. For convenience of notation we define 
\begin{eqnarray}
	\Theta	& :=	& \prod_{i = 1}^n b_i^{\det \lf(M_i\ri)/\det(M)}. \label{EquTheta}
\end{eqnarray}

With the results of the lemmas we are able to establish the main theorem of this section.

\begin{thm} Let $n \geq 2$, let 
$f_\kappa$ be a parametric family of the form (\ref{Equfkappa}) in $\PD$ 
with $\alpha(0) = 0$, $b_0 = 1$,
and let $\Theta$ be defined by (\ref{EquTheta}). Then we have:
\begin{enumerate}
	\item[(a)] For $\kappa = |\Theta|$ we have $\eq(y) = \eq(0) = \cdots = \eq(n)$. Hence, in particular, $\cA(f_\kappa)$ is solid for all choices of $\arg(c)$ whenever $\kappa \leq |\Theta|$.
	\item[(b)] For $\kappa > (n+1) \cdot |\Theta|$ we have $\eq(y) \not\in \cA(f_\kappa)$ and hence $\cA(f_\kappa)$ has genus 1. If additionally $\arg(c)$ is in extreme opposition and the inner complement component $E_y(f_{\kappa})$ appears finally at the point $\eq(y)$ then this bound is sharp, i.e., $\eq(0) \in \cA(f_{(n+1) \cdot |\Theta|})$.
\end{enumerate}
\label{ThmRoughBounds}
\end{thm}
Note that the question to decide \textit{if} the inner complement component appears finally at $\eq(y)$ will be discussed in the next section.

\begin{proof}
As initial preparation, we note that for $f \in \PD$ and any $\mathbf{z} \in \lf(\fC^*\ri)^n$ with $\Log|\mathbf{z}| = \eq(y)$ we have $|m_y(\mathbf{z})| = |c| / |\Theta|$. Namely, by Lemma \ref{LemEquilibPointCoord} we have	
\[
	\lf|m_y\lf(\mathbf{z}\ri)\ri|	\ =	\ |c| \cdot e^{\langle \eq(y), y \rangle}	\ =	\ |c| \cdot \exp\lf(-\lf\langle\lf(M^t\ri)^{-1} \cdot \Log|b|, y\ri\rangle\ri)
\]
and the claim follows with Cramer's rule.

(a) Let $\mathbf{z} \in \lf(\fC^*\ri)^n$ with $\Log|\mathbf{z}| = \eq(y)$. By Lemma \ref{LemEquilibPointCoord} we have $\lf|m_i\lf(\mathbf{z}\ri)\ri| = 1$ for all $i \in \{0,\ldots,n\}$. If $\kappa = |\Theta|$ we have $\lf|m_y\lf(\mathbf{z}\ri)\ri| = 1$ as well due to initial calculation. Hence by definition of $\eq(y)$ and of the $\eq(k)$ all equilibrium points coincide. The solidness of $\cA(f_\kappa)$ for such $\kappa$ follows from Lemma \ref{LemInnerEquiPoint}.

(b) Assume $\eq(y) \in \cA(f_\kappa)$ for some $\kappa > 0$. Then there exists a $\mathbf{z} \in \lf(\fC^*\ri)^n$ with $\Log|\mathbf{z}| = \eq(y)$ and $f_\kappa(\mathbf{z}) = 0$. By the definition of $\eq(y)$ and our initial calculation, we have $|m_y(\mathbf{z})| = \kappa / |\Theta|$ and $|m_i(\mathbf{z})| = 1$, and thus
\begin{eqnarray}
	\frac{\kappa}{|\Theta|} \cdot e^{i \cdot \lf(\arg(c) + \langle \phi, y \rangle \ri)} + 1 + \sum_{j = 1}^n e^{i \cdot \lf(\arg(b_j) + \langle \phi, \alp{(j)} \rangle \ri)} & =	& 0.
\label{Equ1ProofLemRoughBoundaries}
\end{eqnarray}
But since each exponential term has norm 1, this implies
$\kappa \leq |\Theta| \cdot (n+1)$, contradicting the precondition.

Since $\eq(y) \in \conv\{\eq(0),\ldots,\eq(n)\}$ (Lemma \ref{LemInnerEquiPoint} (b)), 
the precondition  
$\eq(y) \notin \cA(f)$ implies $\eq(y) \in E_y(f)$, and thus 
$E_y(f) \neq \emptyset$.

Assume now that the inner complement component $E_y(f_{\kappa})$ appears finally at $\eq(y)$. 
It suffices to show that $\eq(y) \in \cA(f_{(n+1)|\Theta|})$.
If $\arg(c)$ is in extreme opposition then (by definition of an extremal phase) there exists a $\phi \in [0,2\pi)^n$ satisfying \eqref{Equ1ProofLemRoughBoundaries}
with $\arg(c) + \langle \phi, y \rangle = \pi + \arg(b_j) + \langle \phi, \alp{(j)} \rangle $.
Hence, $\Fa[\eq(0),f](\phi) = - \kappa + (n+1) |\Theta|$ and we have $\eq(0) \in \cA(f_{(n+1)\Theta})$.
\end{proof}

Theorem \ref{ThmRoughBounds} yields the following corollary which is a special case of the class treated in \cite{Nisse3}. 

\begin{cor}
Maximally sparse polynomials with simplex Newton polytope have solid amoebas.
\label{KorMaximallysparse}
\end{cor}

\begin{proof}
For $n = 1$, the amoeba $\cA(f)$ of a maximally sparse polynomial $f$
is a single point. For $n \geq 2$ and $f_{\kappa}$ of the form~\eqref{Equfkappa},
Theorem \ref{ThmRoughBounds} (a) yields that $\cA(f_{\kappa})$ is solid for all $\kappa \leq |\Theta|$. Since $|\Theta| > 0$, $\cA(f_{\kappa})$ is in particular solid for $\kappa = 0$, i.e., if $f$ is maximally sparse.
\end{proof}

\section{Points of appearance of the inner complement component and sharp bounds}
\label{SecSharpBounds}

In the previous section we gave a lower and an upper bound for $\cA(f)$ having genus 0 respectively 1 via investigating the fiber $\Fa[\eq(y),f]$. We have seen that if the inner complement component appears finally at $\eq(y)$, then the upper bound gets sharp. In this section we investigate in general where the complement component appears finally and how this point is related to $\eq(y)$.
Based on this, we provide lower and upper bounds partially improving Theorem \ref{ThmRoughBounds} (see a comparison at the end of the section). We show that,  under some extremal condition, the upper bound is tight and the inner complement component appears finally at a unique, explicitly computable minimum $\mathbf{a}(f_{\kappa})$ which happens to coincide with $\eq(y)$ if and only if the inner lattice point is the barycenter of the Newton polytope (Theorems \ref{ThmImprovedLowerBound}, \ref{ThmSharpUpperBound} and Corollary \ref{KorKroNuCoord}). 

As before, let $\Delta$ be a lattice $n$-simplex and $y$ be in the interior
of $\Delta$. Again, we consider the parametric family $f_\kappa$ as 
introduced in (\ref{Equfkappa}). In the first statement we assume that 
$y=0$.

\begin{thm}
 Let $n \geq 2$ and $f_\kappa$ be a parametric family of polynomials in $\PD[0]$ with $f_\kappa := \kappa \cdot e^{i \cdot \arg(c)} + \sum_{i = 0}^n m_i(\mathbf{z}) = 
\kappa \cdot e^{i \cdot \arg(c)} + \sum_{i = 0}^n b_i \cdot \mathbf{z}^{\alp(i)}$.  Let $\mathbf{w} \in \fR^n$ and assume that 
$|m_0\lf(\Log^{-1}|\mathbf{w}|\ri)| \geq \cdots \geq |m_n\lf(\Log^{-1}|\mathbf{w}|\ri)|$. Then there exists a $\kappa \in \fR_{> 0}$ such that 
	\begin{eqnarray*}
		\kappa \ \geq \ \sum_{i = 0}^{n-2} |m_i\lf(\Log^{-1}|\mathbf{w}|\ri)| \ \text{ and } \ \mathbf{w} \not\in E_y(f_\kappa).
	\end{eqnarray*}
\label{ThmImprovedLowerBound}
\end{thm}

\begin{proof}
Since $\alp(0),\ldots,\alp(n)$ form a simplex, there is a dual basis 
$\alp(1)^*,\ldots,\alp(n)^* \in \fQ^n$ with $\langle \alp(j)^*,\alp(k) \rangle = 0$ for all $k \not\in \{j,0\}$. We will choose $\lambda_1, \ldots, 
\lambda_n \in [0, 2\pi)$ such that for 
$\phi := \sum_{j=1}^n \lambda_j \alpha(j)^*$ we get
$\Fa[\mathbf{w},f_\kappa](\phi) = 0$ for some $\kappa \in \fR_{>0}$
sufficiently large.

We can choose $\lam_2,\ldots,\lam_n \in [0,2\pi)$ with
\begin{eqnarray*}
	e^{i \cdot \lf(\arg(b_j) + \langle \lam_j \alp(j)^*,\alp(j)\rangle\ri)}	& =	& \arg(c) + \pi \text{ for all } j \in \{2,\ldots,n\}.
\end{eqnarray*}
We may finally choose $\lam_1 \in [0,2\pi)$ such that the sum of the two shortest monomials
\begin{eqnarray*}
	|m_0\lf(\Log^{-1}|\mathbf{w}|\ri)| \cdot e^{i \cdot \lf(\arg(b_0) + \sum_{k = 1}^n \langle \lam_k \alp(k)^*,\alp(0)\rangle\ri)} + |m_1\lf(\Log^{-1}|\mathbf{w}|\ri)| \cdot e^{i \cdot \lf(\arg(b_1) + \langle \lam_1 \alp(1)^*,\alp(1)\rangle\ri)}
\end{eqnarray*}
is either zero or a complex number with argument $\arg(c) + \pi$, due to the following Rouch\'{e}-type principle from complex
analysis. Recall that the winding number of a closed curve $\gamma$ in the complex plane around a point $z$ is given by $\frac{1}{2\pi i} \int_{\gamma} \frac{d \zeta}{\zeta - z}$.

\smallskip

\noindent
\emph{Claim.} For $A,B \in \fC$ with $A > B$ and $r,s \ge 1$ the function
  $g(\phi) := A \cdot e^{i \cdot r \phi} + B \cdot e^{i \cdot s \phi}$ with $\phi \in [0,2\pi)$ has a non-zero winding number with
  respect to the origin.

\smallskip

Clearly, the function $A \cdot e^{i \cdot r \phi}$ has a non-zero winding number. Now
assuming that $g$ has a winding number of zero, there would exist some
$t \in (0,1)$ such that $h(\phi) := A \cdot e^{i \cdot r \phi} + t \cdot B \cdot e^{i \cdot s \phi}$ has
a zero $\phi$ outside the origin. This is a contradiction.

Altogether, for $\phi := \sum_{j = 1}^n \lam_i \cdot \alp(j)^*$, we get $\Fa[\mathbf{w},f_\kappa](\phi) = ( \kappa - \sum_{j = 0}^{n-2} |m_j\lf(\Log^{-1}|\mathbf{w}|\ri)| + \xi) \cdot e^{i \cdot \arg(c)}$ with $\xi \in \fR_{< 0}$ for $|m_1\lf(\Log^{-1}|\mathbf{w}|\ri)| > |m_0\lf(\Log^{-1}|\mathbf{w}|\ri)|$ and hence $\xi \in \fR_{\leq 0}$ for $|m_1\lf(\Log^{-1}|\mathbf{w}|\ri)| = |m_0\lf(\Log^{-1}|\mathbf{w}|\ri)|$. Thus, we have $\Fa[\mathbf{w},f_{\kappa}](\phi) = 0$ for $\kappa =  |\xi| + \sum_{j = 0}^{n-2} |m_j\lf(\Log^{-1}|\mathbf{w}|\ri)|$. This yields $\mathbf{w} \not\in E_y(f_{\kappa})$ for such choice of $\kappa$.
\end{proof}

Our goal is to characterize the 
$\kappa$ for which the amoeba $\cA(f_\kappa)$ switches the last time from genus 0 to 1.
We first consider the case of $\arg(c)$ in extreme opposition and then use this case
to provide a bound for the general case.

Let $\arg(c)$ be in extreme opposition for $f_{\kappa}$ (note that this property is independent of
the choice of $\kappa$).
For a point $\mathbf{w} \in \Log\lf|\lf(\fC^*\ri)^n\ri|$, the function $\Fa[\mathbf{w},f_{\kappa}]$
from~\eqref{EquFiber} on the fiber of $\mathbf{w}$ evaluates 
for an extremal phase $\phi$ to
\begin{eqnarray*}
	\Fa[\mathbf{w}, f_\kappa](\phi)	& =	& \left( \kappa \cdot e^{\langle \mathbf{w},y \rangle} - 1 - \sum_{j = 1}^n |b_j| \cdot e^{\lf\langle \mathbf{w}, \alp{(j)} \ri\rangle} \right) \cdot e^{i \cdot \psi}
\end{eqnarray*}
for some angle $\psi \in [0,2\pi)$. Since we are only interested in the
zeros of $\Fa[\mathbf{w}, f_\kappa]$, we can always assume $\psi = 0$.
Clearly, $\mathbf{w} \in E_y(f_{\kappa})$ whenever 
$\kappa \cdot e^{\langle \mathbf{w},y \rangle} >  1 + \sum_{j = 1}^n |b_j| \cdot e^{\lf\langle \mathbf{w}, \alp{(j)} \ri\rangle}$.

Since an extremal phase $\phi$ yields the minimal real value of a fiber $\Fa[\mathbf{w}, f_\kappa]$ and since $\cA(f_\kappa)$ has genus 1 if $E_y(f_{\kappa}) \neq \emptyset$, the $\kappa^*$ where $\cA(f_\kappa)$ switches its genus the last time is given by
\begin{eqnarray}
	\min_{\mathbf{w} \in \Log\lf|\lf(\fC^*\ri)^n\ri|}\lf(e^{- \langle \mathbf{w},y \rangle} + \sum_{j = 1}^n |b_j| \cdot e^{\lf\langle \mathbf{w}, \alp{(j)} - y \ri\rangle}\ri) \in \fR_{> 0}\label{GlobalMinimum}.
\end{eqnarray}
The minimizer $\mathbf{w}^*$ then has to be the point $\mathbf{a}(f_{\kappa})$ where the inner complement component finally appears for $\arg(c)$
in extreme opposition, 
since $\mathbf{w}^* \notin E_y(f_{\kappa^*})$, $\mathbf{w}^* \in E_y(f_{\kappa})$ for all $\kappa > \kappa^*$ and for all $\mathbf{w} \neq \textbf{w}^*$ there is a 
$\kappa > \kappa^*$ such that $\mathbf{w} \notin E_y(f_{\kappa})$. 

In the following set $\wh M := \lf(\alp(j)_i - y_i\ri)_{1 \leq i,j \leq n}$ and $\wh M_j$ as the matrix obtained by replacing the $j$-th column of $\wh M$ by $y$.

\begin{lem} Let $\alpha(0) = 0$, $b_0 = 1$, and
$\arg(c)$ be in extreme opposition for $f_{\kappa}$.
The point $\mathbf{a}(f_{\kappa})$ where the inner complement finally appears is
given by $\eq(y) + \mathbf{s}^*$, where $\mathbf{s}^*$ is the solution of the system of linear equations
 \begin{eqnarray}
	M^t \cdot \mathbf{s}	& =	& (\gamma_1,\ldots,\gamma_n)^t 
\label{linequations}
\end{eqnarray}
with $\gamma_j := \log\lf(\det (\wh M_j) / \det (\wh M) \ri)$ 
for $j \in \{1,\ldots,n\}$.
\label{LemSolveMinimum}
\end{lem}

\begin{proof}
It suffices to show that $\eq(y) + \mathbf{s}^*$ solves the problem (\ref{GlobalMinimum}). Substituting $\mathbf{w} = \eq(y) + \mathbf{s}$ 
into~\eqref{GlobalMinimum} and applying Lemma \ref{LemEquilibPointCoord} and Theorem 
\ref{ThmRoughBounds} simplifies the problem to
\begin{eqnarray*}
  |\Theta| \cdot \min_{\mathbf{s} \in \Log\lf|\lf(\fC^*\ri)^n\ri|} \lf(e^{- \langle \mathbf{s},y \rangle} + \sum_{j = 1}^n e^{\lf\langle \mathbf{s}, \alp{(j)} - y \ri\rangle}\ri).
\end{eqnarray*}

To compute the global minimum of $e^{- \langle \mathbf{s},y \rangle} + \sum_{j = 1}^n e^{\lf\langle \mathbf{s}, \alp{(j)} - y \ri\rangle}$ we observe that the partial derivatives
\begin{eqnarray*}
	\frac{\partial f_\kappa}{\partial \, s_i}	& =	& -y_i \cdot e^{- \langle \mathbf{s},y \rangle} + \sum_{j = 1}^n \lf(\alp{(j)}_i - y_i\ri) \cdot e^{\lf\langle \mathbf{s}, \alp{(j)} - y \ri\rangle}
\end{eqnarray*}
vanish if and only if
$\sum_{j = 1}^n \lf(\alp{(j)}_i - y_i\ri) \cdot e^{\lf\langle \mathbf{s}, \alp{(j)}\ri\rangle} = y_i$ for all $i \in \{1,\ldots,n\}$.
We obtain $\wh M \cdot  \lf(e^{\lf\langle \mathbf{s}, \alp{(1)}\ri\rangle},\ldots, e^{\lf\langle \mathbf{s}, \alp{(n)}\ri\rangle}\ri)^t = y,$
and hence $e^{\lf\langle \mathbf{s}, \alp{(j)}\ri\rangle} = \det \wh M_j / \det \wh M$ for $j \in \{1,\ldots,n\}$. Setting 
$\gamma_j := \log \det \wh M_j - \log \det \wh M > 0$ 
yields $\lf\langle \mathbf{s}, \alp{(j)}\ri\rangle = \gamma_j$. Thus, we obtain a system of linear equations~\eqref{linequations}.
Since its solution is unique and $\lim_{|\mathbf{s}| \ra \infty} f(\mathbf{s}) = \infty$ this critical point has to be a minimum.
\end{proof}

Note that, by Lemma \ref{LemEquilibPointCoord} and \ref{LemSolveMinimum}, the point $\mathbf{a}(f_{\kappa})$ is the solution of the linear system
\begin{eqnarray}
	M^t \cdot \textbf{x}	& =	& (\gamma_1 - \log|b_1|,\ldots,\gamma_n - \log|b_n|)^t \label{EquGlobalMinimumCoord}
\end{eqnarray}
and hence may be computed explicitly in terms of the coefficients and exponents of $f$.

\begin{cor}
Let $\arg(c)$ be in extreme opposition for $f_{\kappa}$.
The point $\mathbf{a}(f_{\kappa})$ where the inner complement component appears finally coincides with the equilibrium point $\eq(0)$ if and only if
\[
	\sum_{j = 1}^n \alp{(j)} \ = \ (n + 1) \cdot y.
\]
\label{KorKroNuCoord}
\end{cor}

\begin{proof} Since $b_0 \in \fC^*$ and $\cV(f) \subset (\fC^*)^n$ we may assume $\alpha(0) = 0$, $b_0 = 1$ (otherwise devide $f$ by $b_0 \cdot \mathbf{z}^{\alp(0)}$). Then the result follows from $\sum_{j = 1}^n \lf(\alp{(j)}_i - y_i\ri) \cdot e^{\lf\langle \mathbf{s}, \alp{(j)}\ri\rangle} = y_i$ for all $i \in \{1,\ldots,n\}$.
\end{proof}

With these statements we can prove the main theorem of this section.

\begin{thm}
Let $f_\kappa$ be a parametric family of polynomials in $\PD$ of the form (\ref{Equfkappa}) with $\alpha(0) = 0$, $b_0 = 1$, 
let $\arg(c)$ be in extreme opposition and set
\begin{eqnarray}
	\wh{\Theta}	& =	& \prod_{i = 1}^n \lf(\frac{\det (\wh M) \cdot b_i}{\det\lf(\wh M_i\ri)}\ri)^{\det \lf(M_i\ri)/\det(M)}.
\label{thetahat}
\end{eqnarray}
$\cA(f_{\kappa})$ switches the last time from genus 0 to 1 at
\begin{eqnarray}
 \label{eq:solidbound}
	\kappa	& =	& |\wh{\Theta}| \cdot \lf(1 + \sum_{j = 1}^n \frac{\det \lf(\wh M_j\ri)}{\det (\wh M)}\ri).
\end{eqnarray}
For all other choices of $\arg(c)$ we have: If $\cA(f_{\kappa})$ is solid, then $\kappa$ is strictly 
bounded from above by the right hand side of \eqref{eq:solidbound}.
\label{ThmSharpUpperBound}
\end{thm}

\begin{proof}
Let $\arg(c)$ be in extreme opposition. By Lemma \ref{LemSolveMinimum} it is easy to verify that for an extremal phase $\phi' \in [0,2\pi)^n$ we have $e^{- \lan \mathbf{a}(f_{\kappa}),y \ran} \cdot e^{i \cdot \langle \phi',y\rangle} =  \wh{\Theta}$. We know that $\cA(f_{\kappa})$ 
switches the last time from genus 0 to 1 at
\begin{eqnarray*}
	\kappa^*	&   =	& \min_{s \in \Log\lf|\lf(\fC^*\ri)^n\ri|} \lf(e^{- \langle \eq(y) + s,y \rangle} + \sum_{j = 1}^n |b_j| \cdot e^{\lf\langle \eq(y) + s, \alp{(j)} - y \ri\rangle}\ri) \, .
\end{eqnarray*}
Due to above calculation of $\wh{\Theta}$ and \eqref{EquGlobalMinimumCoord} this is equivalent to~\eqref{eq:solidbound}.

Let $\arg(c)$ be not in extreme opposition. 
We have $E_y(f_\kappa) = \emptyset$ if and only if $\Fa[\mathbf{a}(f_\kappa),f_\kappa] \cap \cV(f_\kappa) \neq \emptyset$. Let $\phi \in [0,2\pi)^n$ be a zero
of $\Fa[\mathbf{a}(f_\kappa),f_\kappa]$. Since $\arg(c)$ is not in extreme opposition, not all outer monomial have the same argument at $\Fa[\mathbf{a}(f_\kappa),f_\kappa](\phi)$ and therefore $|\Fa[\mathbf{a}(f_\kappa),f_\kappa](\phi)| < |\wh{\Theta}| \cdot \lf(1 + \sum_{j = 1}^n \frac{\det \lf(\wh M_j\ri)}{\det (\wh M)} \ri)$.
\end{proof}

It follows from the above derivations that the upper bound for polynomials in $\PD$ to be solid, which we computed in Theorem \ref{ThmSharpUpperBound} improves the upper bound from Theorem \ref{ThmRoughBounds} (b)
in all cases but the one in Corollary~\ref{KorKroNuCoord}.

For the lower bound computed in Theorem \ref{ThmImprovedLowerBound} notice that it holds for all $\kappa$, and hence improves the lower bound from Theorem \ref{ThmRoughBounds} (a), if there exists only one $\kappa$ such that $f_\kappa \in \partial U_y^{A}$ (i.e., if the genus switches only once from 0 to 1 for $\kappa$ running from 0 to $\infty$). 
If this is the case is closely related to the question whether the set $U_y^A$ is connected, which we already mentioned in the introduction to be an open problem.

\section{Lopsidedness and {\it A}-discriminants}
\label{SecLopsidedness}
In the following section we investigate the genus 1 space of amoebas from two other points of view: lopsidedness and $A$-discriminants.

In \cite{Purb1} Purbhoo introduced the concept of \textit{lopsidedness} to provide certificates for points outside of an amoeba 
(see \cite{theobald-wolff-sos} for connections to certificates
by the real Nullstellensatz and sums of squares).
Based on these results and Theorem \ref{ThmSharpUpperBound} we develop a sufficient criterion for amoebas of polynomials in $\PD$ to have genus 1. We recall Purbhoo's main result. Let
$f(\mathbf{z}) = \sum_{i = 1}^d m_i(\mathbf{z}) \in \fC[\mathbf{z}^{\pm 1}]$
be a Laurent polynomial with monomials $m_1,\ldots,m_d$. For a given $\mathbf{w} \in \fR^n$ we define $f\{\mathbf{w}\}$ to be the following sequence of numbers in $\fR_{\geq 0}$:
\begin{eqnarray*}
	f\{\mathbf{w}\}	& :=	& \lf(|m_1(\Log^{-1}|\mathbf{w}|)|,\ldots,|m_d(\Log^{-1}|\mathbf{w}|)|\ri).
\end{eqnarray*}

A sequence of positive real numbers is called \emph{lopsided} if one of the numbers is greater than the sum of all the others. Defining
\begin{eqnarray*}
	\cL\cA(f)	& :=	& \lf\{\mathbf{w} \in \fR^n \ : \ f\{\mathbf{w}\} \text{ is not lopsided}\ri\},
\end{eqnarray*}
it is easy to see that $\cA(f) \subseteq \cL\cA(f)$.

In order to establish a converging hierarchy of approximations of $\cA(f)$, set

\begin{eqnarray*}
	\ti{f}_r(\mathbf{z})	& :=	& \prod_{k_1 = 0}^{r-1} \cdots \prod_{k_d = 0}^{r-1} f\lf(e^{2\pi i k_1 / r} z_1,\ldots,e^{2\pi i k_d / r} z_n\ri) \\
				& =	& \res \lf(\res \lf(\ldots \res(f(u_1 z_1,\ldots,u_d z_d),u_1^r - 1),\ldots, u_{d-1}^r - 1\ri),u_d^r - 1\ri) \, ,
\end{eqnarray*}
where $\res(f,x)$ denotes the resultant with respect to $x$. It is easy to see that $\cA(f) = \cA(\ti{f}_r)$. Then the following theorem holds (see \cite[Theorem 1]{Purb1}).

\begin{thm}
For $n \ra \infty$ the family $\cL\cA(\ti{f}_r)$ converges uniformly to $\cA(f)$. There exists an integer $N$ such that to compute $\cA(f)$ within $\eps > 0$, it suffices to compute $\cL\cA(\ti{f}_r)$ for any $d \geq N$. Moreover, $N$ depends only on $\eps$ and the Newton polytope (or degree) of $f$ and can be computed explicitly from these data.
\label{ThmPurb1}
\end{thm}

As before let $A = \{\alp(1),\ldots,\alp(d)\} \subset \fZ^n$, 
$\fC^A_\Diamond$ be the space of amoebas introduced in 
Section~\ref{SubSecAmoebas}, and $U_{\alp}^A$ be the set of
polynomials $f \in \fC^A_\Diamond$ which have a complement component 
of order $\alp$.

Furthermore, for $f \in \fC^A_\Diamond$ let $\mathbb{T}(f)$ denote the real $d$--torus of polynomials in $\fC^A_\Diamond$ whose coefficients have the same absolute values as the coefficients of $f$, i.e., for $f = \sum_{\alp(j) \in A} b_j \mathbf{z}^{\alp(j)}$ we have $\mathbb{T}(f) := \{\sum_{\alp(j) \in A} e^{i \cdot \psi_j} \cdot b_j \mathbf{z}^{\alp(j)} \ : \ \psi_j \in [0,2\pi) \text{ for all } j\}$.

It is an easy consequence of the definition of lopsidedness that the following
proposition holds (which is, to the best of our knowledge, surprisingly nowhere mentioned in the literature).
\begin{prop}
Let $f = \sum_{\alp(j) \in A} b_j \mathbf{z}^{\alp(j)}$. Assume that $E_{\alpha(1)}(f)$ is non-empty and that there exists some
 $\mathbf{w} \in E_{\alp(1)}(f)$ such that $f\{\mathbf{w}\}$ is lopsided. 
Then $g\{\mathbf{w}\}$ is lopsided for every $g \in \mathbb{T}(f)$. In particular $\mathbb{T}(f) \subset U_{\alp(1)}^A$.
\label{PropLopsidedness}
\end{prop}

\begin{proof}
Since $g\{\mathbf{w}\} = f\{\mathbf{w}\}$ for every $g \in \mathbb{T}(f)$,
for every $\mathbf{w} \in E_{\alp(1)}(f)$ with $f\{\mathbf{w}\}$ lopsided 
we have $g\{\mathbf{w}\}$ lopsided as well. Then, in particular,
$E_{\alp(1)}(g) \neq \emptyset$, whence $g \in U_{\alp(1)}^A$.
\end{proof}

Theorem~\ref{th:lopsidedconnection} shows that for
polynomials in $\PD$ the converse 
is also true. In this statement it is convenient to have 0
as the interior lattice point, so that we set
$A := \{\alp(0),\ldots,\alp(n),0\}$. We may always assume that this is the case, by dividing $f$ by $\mathbf{z}^y$.
 
\begin{thm} \label{th:lopsidedconnection}
Let $f_c = c + \sum_{j = 0}^n b_j \mathbf{z}^{\alp(j)} = 
c + \sum_{j = 0}^n m_j(\mathbf{z})$ be a parametric 
family in $\PD[0]$ with complex parameter $c$, and let 
$\mathbf{a} := \mathbf{a}(f_{|c|})$ be the point where the 
inner complement component appears
finally for positive real parameter values and $\arg(c)$ in extreme
opposition.
If there exists some $d \in \fC^*$ such that 
$\mathbb{T}(f_d) \subset U_0^A$ then
$f_{d}\{\mathbf{a}\}$ is lopsided with 
$|d|$ as the maximal term.
\label{ThmEquivPerforLopsided}
\end{thm}

\begin{proof}
Let $d \in \fC^*$ with $\mathbb{T}(f_d) \subset U_0^A$.
First we show that for every $c \in \fC$ with $|c| \ge |d|$ 
the amoeba $\cA(f_c)$ is of genus~1.

The parametric family $f_c$ forms a complex line in $\PD[0]$ which
can be interpreted as a real plane $H$. By a
result of Rullg{\aa}rd (\cite[Theorem 14]{Rull1}, see also \cite{Mikh1}), 
the intersection of $(U_{\alp}^A)^c$ with an arbitrary projective line 
in $\fC^A_\Diamond$ (viewed as projective space)
is non-empty and connected (even for arbitrary $A$).
For the parameter value $c = 0$ we are in the maximally sparse case, 
and thus Corollary \ref{KorMaximallysparse} implies $f_0 \in (U_0^A)^c$.
By the precondition $\mathbb{T}(f_{d}) \subset U_{0}^A$, the
set $C := \{f_c : c = |d| \cdot e^{i \cdot \phi}, \phi \in [0,2\pi)\} 
\subset \mathbb{T}({f_{d}})$ is contained in $U_0^A$. 
Considered in the plane $H$,
the set $C$ is a circle around the origin. Now the connectedness result
implies that for $|c| \ge |d|$ the amoeba $\cA(f_c)$ is of genus~1 (see Figure~\ref{FigureLopsidedness} for an illustration).

For $\arg(c)$ in extreme opposition,
let $\kappa^* \in \fR$ be the value where $\cA(f_{|c|})$ switches 
the last time from genus 0 to 1. By Theorem~\ref{ThmSharpUpperBound},
the upper bound is attained at some point $\mathbf{z} \in (\fC^*)^n$ 
with $\Log|\mathbf{z}| = \mathbf{a}$ and extremal phase $\phi$. Hence, by evaluating the fiber function of $\mathbf{a}$ at $\phi$
we obtain $\kappa^* =  \sum_{j = 0}^n |b_j| \cdot e^{\lf\lan \mathbf{a}, \alp{(j)} \ri\ran}$.  The auxiliary statement
yields that $\kappa^* < |d|$, and thus
$|d|	> \sum_{j = 0}^n |b_j| \cdot e^{\lf\lan \mathbf{a}, \alp{(j)} \ri\ran}
= \sum_{j = 0}^n \lf|m_j\lf(\Log^{-1}|\mathbf{a}|\ri)\ri|$.
\end{proof}

\ifpictures
\begin{figure}
	\begin{center}
	\begin{picture}(150,130)(0,0)
		\put(0,-20){\includegraphics[width=150pt]{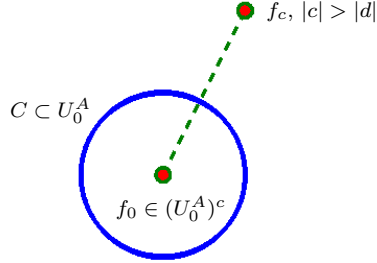}}
		\put(112,100){\tiny{\mbox{$f_c$, $|c| > |d|$}}}
		\put(15,62){\tiny{\mbox{$C \subset U_0^A$}}}
		\put(55,25){\tiny{\mbox{$f_0 \in (U_0^A)^c$}}}
	\end{picture}
	\caption{The real plane $H$ in the proof of 
  Theorem~\ref{th:lopsidedconnection}.}
	\label{FigureLopsidedness}
	\end{center}
\end{figure}
\fi

We recall some of the terminology for $A$-discriminants:
Let $\nabla_0 \subset (\fC^*)^A$ denote the set of all polynomials $f$ such that there exists a $\mathbf{z}^* \in (\fC^*)^n$ with
\begin{eqnarray*}
	f(\mathbf{z}^*)	\ =	\ \frac{\partial f}{\partial z_1}(\mathbf{z}^*)	\ =	\ \cdots	\ =	\ \frac{\partial f}{\partial z_n}(\mathbf{z}^*)	\ =	\ 0
\end{eqnarray*}
and let $\nabla_A$ denote the Zariski closure of $\nabla_0$. 
If the variety $\nabla_A$ is of codimension 1, then the $A$-\emph{discriminant} $\Delta_A$ is defined as the irreducible, integral polynomial in the coefficients $b_1,\ldots,b_d$ of $f \in (\fC^*)^A$ as variables
which vanishes on $\nabla_A$. 
The $A$-discriminant is unique up to sign (see \cite[Chapter 9, p. 271]{GelKapZel}).

The following theorem shows that, for polynomials in $\PD$, there is a strong connection between their $A$-discriminants and the topology of their amoebas. Here, $\ovl{U_y^A}$ denotes the topological closure of the 
set $U_y^A$. Set $A := \{\alpha(0), \ldots, \alpha(n),y\}$.

\begin{thm}  \label{th:adisc1}
Let $\alpha(0) = 0$, $b_0 = 1$. A polynomial 
$f = c \cdot \mathbf{z}^y + 1 + \sum_{i = 1}^n b_i \cdot \mathbf{z}^{\alp{(i)}}$
is contained in $\nabla_A$ if and only if the expression
\begin{eqnarray}
	c	+  \wh{\Theta} \cdot \lf(1 + \sum_{j = 1}^n \frac{\det \lf(\wh M_j\ri)}{\det (\wh M)} \ri) 
\label{cpluswhtheta}
\end{eqnarray}
in the variables $b_1,\ldots,b_n,c$ vanishes. Here, 
$\wh{\Theta}$ is defined as in \eqref{thetahat}.
\end{thm}

Note that a power of the summands of~\eqref{cpluswhtheta} is a binomial.

\begin{cor} \label{co:adisc1}
Let $\alpha(0) = 0$ and $b_0 = 1$. 
The $A$-discriminant $\Delta_A$ is a binomial whose variety coincides
with the set of projective points $(1: b_1 : \ldots: b_n:c)$ where $\arg(c)$ is
in extreme opposition and $\cA(f_{|c|})$ switches the last time
from genus~0 to genus~1 exactly at the value~ 
$|c| = |\wh{\Theta}| \cdot (1 + \sum_{j = 1}^n \det \wh M_j / \det \wh M)$.
\end{cor}

Note that a power of the summands of~\eqref{cpluswhtheta} is an
irreducible binomial with rational coefficients. Up to normalizing
the coefficients, this is the $A$-discriminant.

\medskip

\noindent
\emph{Proof of Theorem~\ref{th:adisc1}.}
For the given polynomial $f \in \PD$. we have
\begin{eqnarray}
	\frac{\partial f}{\partial z_j}	& =	& y_j \cdot c \cdot \mathbf{z}^{y - e_j} + \sum_{i = 1}^n b_i \cdot \alp(i)_j \cdot \mathbf{z}^{\alp{(i)} - e_j} \label{EquProofKorConfigSpace2} \, , \quad 1 \le j \le n,
\end{eqnarray}
where $e_j$ denotes the $j$-th unit vector. Assume that arbitrary $b_1,\ldots,b_n \in \fC^*$ are fixed. Substituting $f$ into $\mathbf{z}^{e_j}$ times (\ref{EquProofKorConfigSpace2}) yields a regular system of linear equations in $(\mathbf{z}^{\alp(1)},\ldots, $ $\mathbf{z}^{\alp(n)})$. The regularity comes from the fact that the $\alp(1),\ldots,\alp(n)$ are the vertices of a simplex.  Hence there are
only finitely many solutions $\mathbf{z}^* \in (\fC^*)^n$ such that all partial derivatives vanish, and all of these solutions have the same norm. For any 
such solution $\mathbf{z}^*$, solving $f = 0$ for $c$ yields a unique and non-zero $c$ such that the $f$ corresponding to the coefficients $b_1,\ldots,b_n,c$ is in $\nabla_A$. This argumentation shows furthermore that $\nabla_A$ is a subvariety of codimension 1 and hence $\Delta_A$ exists. Observe that $\mathbf{z}^*$ does not depend on $c$.

Let now $\phi' \in [0,2\pi)^n$ be an extremal phase. Then $(\mathbf{a}(f_\kappa), \phi') = \mathbf{z}^*$ since we know $\frac{\partial f}{\partial z_j}(\mathbf{a}(f_\kappa), \phi') = 0$ for all $j \in \{1,\ldots,n\}$ from the last section (see the proof of Lemma \ref{LemSolveMinimum}). But since further $\Fa[\mathbf{a}(f_\kappa),f](\phi') = 0$ if and only if $c$ is in extreme opposition and its norm equals the bound from Theorem \ref{ThmSharpUpperBound}, we have $f \in \nabla_A$ if and only if~\eqref{cpluswhtheta} vanishes.
\hfill $\Box$

\medskip

\noindent
\emph{Proof of Corollary~\ref{co:adisc1}.} Expression
\eqref{cpluswhtheta} is a Laurent binomial in the variables 
$b_1,\ldots,b_n,c$ with rational coefficients and monomials in distinct 
variables. Now the statement follows from Theorem~\ref{th:adisc1} 
via Theorem~\ref{ThmSharpUpperBound}.
\hfill $\Box$

\medskip

We remark that a different connection between $A$-discriminants and amoebas was
investigated by Passare, Sadykov and Tsikh \cite{PaSaTsi1} who studied the amoebas of
$A$-discriminantal hypersurfaces. For further connections between $A$-discriminants and polynomials in the class $\PD$, see also \cite{Bihan,Bihan:Rojas:Stella}.

\section{The barycentric case}
\label{SecBarycenter}
In this section we treat polynomials in $\PD$ where the exponent of the inner monomial is the barycenter of the simplex spanned by the exponents of the outer monomials. We call such a pair $(\Delta, y)$ \emph{barycentric}.
For this class we provide a complete classification of the space of amoebas, i.e. the set $U_y^A$ and its complement $(U_y^A)^c$. In particular, we are able to answer Rullg{\aa}rd's question for this barycenter case by showing that set $U_y^A$ is path-connected (Corollary \ref{KorJackpot}).

In \cite[Proposition~2]{PR1} Passare and Rullg{\aa}rd showed that the amoeba
of $f(\mathbf{z}) := 1 + c \cdot z_1 \cdots z_n + \sum_{i = 1}^n z_i^{n+1} \in \fC[\mathbf{z}]$
has a complement component of order $(1,\ldots,1)$ if and only if $0 \not\in \cA(f)$. Moreover, this component exists if and only if
$c \not\in \{-t_1 - \cdots - t_n : t_i \in \fC, |t_i| = 1, t_1 \cdots t_n = 1\}$. We generalize this result as well as our Corollary \ref{KorKroNuCoord} to the following theorem. From now on let $n \ge 2$,
$A = \{\alp(0),\ldots,\alp(n),y\}$ and $\conv A = \Delta$.

\begin{thm}
Let $(\Delta,y)$ be barycentric, and let
$f_c$ be a family of parametric polynomials in $\PD$ with parameter 
$c \in \fC$ (i.e., $|c|$ and $\arg(c)$).
Then for every parameter value $c \in \fC$ 
the following statements are equivalent:
\begin{enumerate}
	\item [(a)] $f_c \in U_y^{A}$ (i.e. $\cA(f_c)$ has genus 1),
	\item [(b)] $\eq(y) \in E_y(f_c)$,
	\item [(c)] $c \not\in \lf\{-|\Theta| \cdot \sum\limits_{j = 0}^n e^{i \cdot (\arg(b_j) + \langle \alp(j) - y, \phi \rangle) } : \phi \in [0,2\pi)^n\ri\}.$
\end{enumerate}
\label{ThmMain2}
\end{thm}

Note that (c) generalizes the condition from the example above, since if all coefficients of $f_c$
are 1 then $\Theta = 1$.

\begin{proof}
Since the inner lattice point $y$ is the barycenter, we have $f_c = c \cdot \mathbf{z}^y + \sum_{j = 0}^n b_j \cdot \mathbf{z}^{\alp(j)}$ and $\sum_{j = 0}^n \alp(j) = (n+1) \cdot y$. As usual, we may assume $b_0 = 1$ and $\alp(0) = 0$.

(b) $\Lera$ (c): Since $\alp(0),\ldots,\alp(n)$ form a simplex, the equilibrium point $\eq(y)$ is unique. At $\eq(y)$ we have for the outer monomials $|b_i| \cdot e^{\lan \alp(i),\eq(y) \ran} = 1$ (Definition \ref{DefEquilibriumPoint}) and furthermore $e^{\lan y,\eq(y) \ran} = 1/|\Theta|$ (proof of Theorem \ref{ThmRoughBounds}). Hence, at $\eq(y)$ the fiber function is given by
\begin{eqnarray*}
	\Fa[\eq(y),f_c](\phi)	& =	& c \cdot e^{i \cdot \langle y, \phi \rangle} + |\Theta| \cdot \sum_{j = 0}^n e^{i \cdot (\arg(b_j) + \langle \alp(j), \phi \rangle)}.	
\end{eqnarray*}
Thus, if and only if the condition (c) is satisfied, the zero set
$\cV(\Fa[\eq(y),f_c])$ of the fiber function $\Fa[\eq(y),f_c]$
is empty and therefore $\eq(y) \in \fR^n \bs \cA(f_c)$. Since by Theorem \ref{ThmRoughBounds} (a) $\eq(y)$ may be contained in the complement of $\cA(f_c)$ only if $c$ is the dominant term, we have with Lemma \ref{LemRullgard} that $\eq(y) \in \fR^n \bs \cA(f_c)$ if and only if $\eq(y) \in E_y(f_c)$.

(b) $\Ra$ (a) is trivial. (a) $\Ra$ (b): Since we are only interested in $\cV(f_c)$ we may normalize such that $y = 0$ and hence $\sum_{i = 1}^n \alp(i) = -\alp(0)$. We show that $\cA(f)$ is symmetric around $\eq(y)$: Assume that $\eq(y) + \mathbf{w} \in E_y(f_c)$ for an arbitrary $\mathbf{w} \in \fR^n$. Setting $\lam_j = \langle \alp(j), \eq(y) + \mathbf{w}\rangle$ 
for $j \in \{1,\ldots,n\}$ we obtain
\begin{eqnarray}
\label{centercalculation}
	\langle \alp(0), \eq(y) + \mathbf{w}\rangle \ = \ -\sum_{j = 1}^n \langle \alp(j), \eq(y) + \mathbf{w} \rangle
		\ = \ -\sum_{j = 1}^n \lam_j.
\end{eqnarray}
Then for any permutation of the $\lam_j$ there is a $\mathbf{w}'$ with $\langle \alp(j), \eq(y) + \mathbf{w}'\rangle = \lam_j$ for $j \in \{0,\ldots,n\} \bs \{k,l\}$ and $\langle \alp(k), \eq(y) + \mathbf{w}'\rangle = \lam_l$, $\langle \alp(l), \eq(y) + \mathbf{w}'\rangle = \lam_k$. This is obvious for $k,l \in \{1,\ldots,n\}$. Thus, let $k = 0, l = 1$ and $\langle \alp(0), \eq(y) + \mathbf{w}'\rangle = \lam_1$. Then by~\eqref{centercalculation} we have
$$\langle \alp(1),\eq(y) + \mathbf{w}' \rangle \ = \ - \langle \alp(0),\eq(y) + \mathbf{w}' \rangle - \sum_{j = 2}^n \langle \alp(j),\eq(y) + \mathbf{w}' \rangle \ = \ - \sum_{j = 1}^n \lam_j,$$
i.e., every permutation of the lengths of the monomials at $\eq(y) + \mathbf{w}$ is realized at some point $\eq(y) + \mathbf{w}'$. Similarly, let $\phi \in [0,2\pi)^n$ with $\exp(i \cdot \langle \alp(j), \phi \rangle) = \psi_j \in [0,2\pi)$. Then, with the same argumentation, there is a $\phi'$ realizing every given permutation of the $\psi_j$. Altogether, such a permutation is realized by some $\fC^*$-basis transformation
on $(\fC^*)^n$.
Thus, if $\mathbf{w}'$ realizes some permutation of the $\lam_j$, then there 
exists an automorphism $\pi$ on $[0,2\pi)^n$ such that for all $\psi \in [0,2\pi)^n$: $\Fa[\eq(y) + \mathbf{w},f_c](\psi) = \Fa[\eq(y) + \mathbf{w}',f_c](\pi(\psi))$. Hence, we have for all such $\mathbf{w}'$:
\begin{eqnarray}
	\eq(y) + \mathbf{w} \in E_y(f_c)	& \Ra	& \eq(y) + \mathbf{w}' \in E_y(f_c). \label{EquSymmetry}
\end{eqnarray}

Now investigate the complement-induced tropical hypersurface 
$\cC(f_c - c)$ (see Section \ref{SubSecSpine}) with $\eq(y)$ as unique vertex. 
Let $A_0,\ldots,A_n$ denote the cells given by the decomposition $\fR^n \bs \cC(f_c - c)$. Since $E_y(f_c)$ is an open set and $\cC(f_c - c)$ has codimension one in $\fR^n$, we can assume that $\mathbf{w}$ is contained in the interior of some $A_i$.
The fact that every permutation of the $\lam_i$ is realized at some point $\eq(y) + \mathbf{w}'$ together with \eqref{EquSymmetry} yields: If $\eq(y) + \mathbf{w} \in A_i$ then there is some $\eq(y) + \mathbf{w}' \in E_y(f_c)$ for every $A_j \neq A_i$. Since $\eq(y)$ is the unique vertex of $\cC(f_c - c)$ and due to convexity of $E_y(f_c)$ this implies for every $\mathbf{w} \neq 0$
\begin{eqnarray*}
	\eq(y) + \mathbf{w} \in E_y(f_c)	& \Ra	& \eq(y) \in E_y(f_c).	
\end{eqnarray*}
\vspace*{-6ex}

\end{proof}

Theorem \ref{ThmMain2} yields that understanding $U_y^A$ and its complement can be reduced to understanding the fiber function $\Fa[\eq(y),f_c]$ and its variety. With this approach we will be able to provide a geometric description of $U_y^A$ and $(U_y^A)^c$.

For $R > r$, a hypocycloid with parameters $R,r$ is the parametric curve in $\fR^2 \cong \fC$ given by
\begin{eqnarray}
	[0,2\pi) \ra \fC, \quad \phi \mapsto (R - r) \cdot e^{i \cdot \phi} + r \cdot e^{i \cdot \lf(\frac{r - R}{r}\ri) \cdot \phi}.
	\label{EquHypocycloidDef}
\end{eqnarray}
Geometrically, it is the trajectory of some fixed point on a circle with radius $r$ rolling (from the interior)
on a circle with radius $R$. The main part of this section is devoted
towards proving the following nice and explicit characterization
of~$\partial (U_y^{A})^c$.

\begin{thm}
Let $(\Delta,y)$ be barycentric.
For given $b_0,\ldots,b_n \in \fC^*$ the intersection of the set $\partial (U_y^{A})^c$ with the 
complex line $\{(b_0, \ldots, b_n, c) : c \in \fC\}$ is given by the (eventually rotated) hypocycloid with 
parameters $R = (n+1) \cdot |\Theta|$, $r = |\Theta|$ and with cusps at 
\begin{eqnarray}
	\arg(c)	& =	& \pi \cdot \lf(1 + \frac{2k -\sum_{i = 1}^n \arg(b_i)}{n + 1}\ri), \quad k \in \{0,\ldots,n\}.
	\label{EquHypocycloid}
\end{eqnarray}
\label{ThmHypocycloid}
\end{thm}

We have already seen that it suffices to treat the case $y = 0$. Let $f_c \in \PD[0]$ be a parametric family with $\sum_{i = 0}^n \alp(i) = 0$ and fixed $b_0,\ldots,b_n \in \fC^*$, $b_0 = 1$. For $f_c$ consider the set
\begin{eqnarray}
	S	& :=	& \lf\{c \in \fC : \cV(\Fa[\eq(y),f_c]) \neq \emptyset\ri\} \label{EquSetS}
\end{eqnarray}
as a subset of $\fR^2 \cong \fC$. Theorem \ref{ThmMain2} shows that $S$ is exactly the set of all $c \in \fC$ such that the inner complement component of $\cA(f)$ exists. Hence, $S \subseteq \fR^2$ is located in the space $\PD[0]$ intersected with the complex line $\{(b_0,\ldots,b_n,c) : c \in \fC\}$ induced by the family $f_c$. It contains all coefficient vectors of polynomials not belonging to $U_y^{A}$. As a first step towards the proof of Theorem~\ref{ThmHypocycloid}
we show a technical result on the set $S$.

\begin{lem}
Let $k := -n + 1 + (-1)^{n+1}$ and
\begin{eqnarray}
	F: [k,n] \times [0,2\pi) \ra \fC, \quad (\mu,\psi) \mapsto |\Theta| \cdot \mu \cdot e^{i \cdot \psi} + |\Theta| \cdot e^{i \cdot (- n \cdot \psi + \sum_{j = 1}^n \arg(b_j))} \label{EquSimpleFiberDescription}.
\end{eqnarray}
Then
\begin{enumerate}
	\item The image of $F$ is contained in the set $S$ defined in \eqref{EquSetS}.
	\item Up to a rotation, the curve parameterized by $\phi \mapsto F(n,\phi)$ for $\phi \in [0,2\pi)$ is a hypocyloid \eqref{EquHypocycloidDef} with $R = (n+1) \cdot |\Theta|$, $r = |\Theta|$.
\end{enumerate}
\label{LemConstrF}
\end{lem}

\begin{proof}
By \eqref{EquSetS} and \eqref{EquFiber} the set $S$ is given by the image of the function $g: [0,2\pi)^n \ra \fC, \phi \mapsto -|\Theta| \cdot \sum_{j = 0}^n e^{i \cdot (\arg(b_j) + \langle \alp(j), \phi \rangle)}$ (Theorem \ref{ThmMain2}). The idea of the proof is that the image of $g$ restricted to some 
particular subset of $[0,2\pi)^n$ is exactly the image of $F$.

Let again $\alp(1)^*,\ldots,\alp(n)^* \in \fQ^n$ denote the dual basis of $\alp(1),\ldots,\alp(n)$, and set  
$$h(\phi) \ := \ g(\phi) - |\Theta| \cdot e^{i \cdot \langle \alp(0),\phi \rangle} \ = \ -|\Theta| \cdot \sum_{j = 1}^n e^{i \cdot (\arg(b_j) + \langle \alp(j), \phi \rangle)}.$$
Further let $\psi \in [0,2\pi)$ and $\sig_\psi$ denote the segment $[-k \cdot |\Theta| \cdot e^{i \cdot \psi},n \cdot |\Theta| \cdot e^{i \cdot \psi}] \subset \fC$.

We first discuss the case of $n$ even. For fixed $\psi$, let $M := \{\phi_\xi : \xi \in [0,\pi]\}$ with
\begin{eqnarray*}
	\phi_\xi	& :=	& \sum_{j = 1}^{n/2} (\psi - \arg(b_j) + \xi) \cdot \alp(j)^* + \sum_{j = n/2 + 1}^n (\psi - \arg(b_j) - \xi) \cdot \alp(j)^*.
\end{eqnarray*}
Since $\arg(b_j) + \langle \alp(j), \phi_\xi \rangle = \psi + \xi$ for $j \leq n/2$ (resp. $\psi - \xi$ for $j > n/2$) and since all summands have norm $|\Theta|$, we see $e^{i \cdot (\arg(b_j) + \langle \alp(j),\phi_\xi \rangle)} + e^{i \cdot (\arg(b_j) - \langle \alp(j + n/2),\phi_\xi \rangle)} \in \sig_\psi$. Thus, $h(\phi_\xi) \in \sig_\psi$ for all $\phi_\xi \in M$.

Since furthermore the real part of $h(\phi_\xi) \cdot e^{-i \cdot \psi}$ is given by $n \cdot \cos(\xi)$, the image of $h(M)$ is $\sig_\psi$, i.e. $\{|\Theta| \cdot \mu \cdot e^{i \cdot \psi} : \mu \in [k,n]\}$. Finally 
we have for every $\phi_\xi \in M$
\begin{eqnarray*}
	\langle \alp(0),\phi_\xi \rangle
	  & =	& \big\langle - \sum_{j = 1}^n \alp(j),\phi_\xi \big\rangle 
	  \ =	\ \sum_{j = 1}^{n/2} \arg(b_j) - \psi + \xi + \sum_{j = n/2 + 1}^{n} \arg(b_j) - \psi - \xi \\
	  & =	& \sum_{j = 1}^{n} \arg(b_j) - \psi.
\end{eqnarray*}
Hence the set $g(M) = \{h(\phi_\xi) + |\Theta| \cdot e^{i \cdot \langle \alp(0),\phi_\xi \rangle} : \phi_\xi \in M\}$ coincides with the set
$\{|\Theta| \cdot (\mu \cdot e^{i \cdot \psi} + e^{i \cdot (\sum_{j = 1}^{n} (\arg(b_j) - \psi))}) : \mu \in [k,n])\}$, i.e., $g(M) = F([k,n],\psi)$.

If $n$ is odd, the argument is analogous up to the fact that we redefine $M := \{\phi_\xi : \xi \in [0,\pi]\}$ by 
\begin{eqnarray*}
	\phi_\xi	& :=	&(\psi - \arg(b_1)) \cdot \alp(1)^* + \sum_{j = 2}^{\lceil n/2 \rceil} (\psi - \arg(b_j) + \xi) \cdot \alp(j)^* + \\
			&	& \sum_{j = \lceil n/2 \rceil + 1}^{n} (\psi - \arg(b_j) - \xi) \cdot \alp(j)^*.
\end{eqnarray*}
This proves the first statement.

For the choice of $R$ and $r$ we obtain the hypocycloid curve $\{|\Theta| \cdot n \cdot e^{i \cdot \phi} + |\Theta| \cdot e^{-i \cdot n |\Theta| \phi} \, : \, \phi \in [0,2\pi)\}$, which coincides with the image of $F(n,\psi), \psi \in [0,2\pi)$ up to a coordinate change  given by $\psi \mapsto \lf(\frac{\sum_{i = 1}^n \arg(b_i)}{n + 1}\ri) + \phi$. This is the second statement.
\end{proof}

Indeed, the next lemma states that the set $S$ defined in \eqref{EquSetS} \emph{exactly} coincides 
with the region defined by the hypocycloid curve. See the Appendix for a detailed
calculation.

\begin{lem} \label{le:hypocycloid2}
The set $S$ equals the region $T$ whose boundary is (up to rotation) the hypocycloid with parameter $R = (n+1) \cdot |\Theta|$, $r = |\Theta|$ given by $\phi \mapsto F(n,\phi)$ for $\phi \in [0,2\pi)$. In particular, $S$ is simply connected.
\label{LemSimplyConnected}
\end{lem}

With these results we are able to prove Theorem \ref{ThmHypocycloid}:

\medskip

\noindent
\emph{Proof of Theorem \ref{ThmHypocycloid}.}
Again, we may assume that $y$ is the origin. 
For $b_0,\ldots,b_n \in \fC^*$ we investigate the parametric family
$f_c = c + \sum_{j = 0}^n b_i \cdot \mathbf{z}^{\alp(i)} \in \PD[0]$
with a parameter $c \in \fC$. On this complex line in the space of amoebas
we want to describe $\partial (U_y^{A})^c$. 

By Theorem \ref{ThmMain2} (c), $\cA(f_c)$ has genus 1 if and only if $c \not\in \{|\Theta| \cdot \sum_{j = 0}^n e^{i \cdot (\arg(b_j) + \langle \alp(j) - y, \phi \rangle) } : $ $\phi \in [0,2\pi)^n\}$ (recall that $\Theta$ depends on the choice of the $b_i$) which is the complement of $S$ by definition. Therefore
\begin{eqnarray*}
	\partial \lf((U_y^{A})^c\ri) \cap \{(b_0,\ldots,b_n,c) : c \in \fC\}	& =	& \partial S.
\end{eqnarray*}
By Lemma \ref{LemSimplyConnected}, $\partial S$ is up to rotation a hypocycloid with parameters $R = (n+1) \cdot |\Theta|$, $r = |\Theta|$ around the origin. The location of the cusps follows from the definition of the $\partial S$-describing function $F$ in \eqref{EquSimpleFiberDescription} solving $i \cdot \lam = -i \cdot (n\cdot \lam + \sum_{j = 1}^n \arg(b_j)) \mod 2 \pi$.
\hfill $\Box$

\begin{exa}
For the parametric family of polynomials $f_c = 1 + 2.4 \cdot z_1^2z_2 + c \cdot z_1 z_2^3 + (1 + 1.3i) \cdot z_1 z_2^8$, the set $\PD[y] \cap \{(1:2.4:1+1.3 \cdot i:c) : c \in \fC\}$ is illustrated in Figure \ref{FigureHypocycloid}. The non-real choice of one of the ``outer'' coefficients causes a rotation of the set as described in Theorem \ref{ThmHypocycloid}.
\end{exa}

\ifpictures
\begin{figure}[ht]
	\begin{center}
		\includegraphics[width=0.5\linewidth]{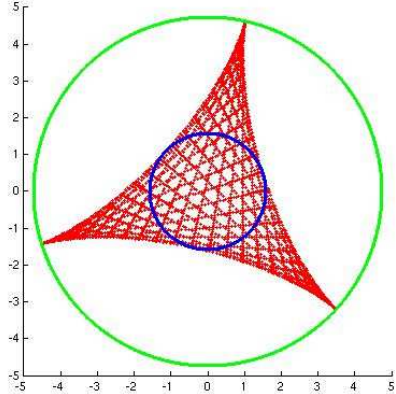}
	\caption[A meshplot of $\cV(\cF_{\eq(y),f})$.]{A meshplot of $S$. The green (light) circle has radius $3 \cdot |\Theta|$ and the blue (dark) circle has radius $|\Theta|$ with $|\Theta| \approx 1.5789$.}
	\label{FigureHypocycloid}
	\end{center}
\end{figure}
\fi

Finally, we show path-connectivity of the set $U_y^A$ and therefore answer Rullg{\aa}rd's question for all spaces of amoebas of polynomials with barycentric simplex Newton polytopes with one inner lattice point (see Corollary \ref{KorJackpot}). As a cornerstone, we show the following general result about spaces of amoebas.

\begin{thm}
Let $A = \conv\{\alp(1),\ldots,\alp(d)\}$ and $j \in \{1, \ldots, d\}$.
If for every $b \in \fC^A_\Diamond$ the set $\{(b_1,\ldots,b_d) : b_j \in \fC^*\} \cap \lf(U_{\alp(j)}^A\ri)^c$ is simply connected, then $U_{\alp(j)}^{A}$ is path-connected.
\label{ThmPathconstruction}
\end{thm}

\begin{proof}
We identify $b \in \fC^A_\Diamond$ with $f_b := \sum_{i = 1}^d b_i \cdot \mathbf{z}^{\alp(i)} \in \fC^A_\Diamond$. Since no assumptions are made about the $\alp(j)$ here we may choose $j = 1$ to abbreviate notation. Let $a, b \in U_{\alp(1)}^{A} \subseteq \fC^A_\Diamond$. We construct an explicit path $\gamma$ between $a$ and $b$ such that $\gamma \in U_{\alp(1)}^A$. Let $[a,b]$ denote the line segment $a + \mu \cdot (b - a) \subset \fC^A_\Diamond, \mu \in [0,1]$. For the construction of the path we need a value $\kappa \in \fR_{> 0}$ for the norm of the first coordinate of points in $\fC^A_\Diamond$ such that every point on $[a,b]$ is lopsided. This is guaranteed by
\begin{eqnarray}
	\kappa	& :=	& 1 + \max_{c \, \in \, [a,b]} \min_{\mathbf{w} \, \in \, \fR^n} \Big\{\sum_{i = 2}^d |c_i| \cdot e^{\langle \mathbf{w},\alp(i) - \alp(1)\rangle}\Big\} \in \fR_{>0}. \label{EquThmPathConstr1}
\end{eqnarray}
Define the points $a', b' \in \fC^A_\Diamond$ by
$$
	a' \ := \ (\kappa \cdot \arg(a_1), a_2 , \ldots , a_d), \quad b' \ := \ 
                  (\kappa \cdot \arg(b_1), b_2 , \ldots , b_d).
$$
The choice of $\kappa$ guarantees that the polynomials $f_{a'}$ and $f_{b'}$ are lopsided at some point with the monomial with exponent $\alp(1)$ as dominant term and therefore $a', b' \in U_{\alp(1)}^{A}$. Since for every $b \in \fC^A_\Diamond$ the set $\{(b_1,\ldots,b_d) : b_1 \in \fC\} \cap (U_{\alp(1)}^A)^c$ is simply connected and since $a, a', b, b' \in U_{\alp(1)}^A$, there exists a path $\gamma_1$ from $a$ to $a'$ and a path $\gamma_2$ from $b'$ to $b$ with $\gamma_1 \subset \{(a_1,a_2,\ldots,a_d) : a_1 \in \fC\} \cap U_{\alp(1)}^A$ and $\gamma_2 \subset \{(b_1,b_2,\ldots,b_d) : b_1 \in \fC\} \cap U_{\alp(1)}^A$. Let
\begin{eqnarray*}
	d \ := \ (\kappa \cdot \arg(b_1), \arg(b_2) \cdot |a_2| , \ldots , \arg(b_d) \cdot |a_d|).
\end{eqnarray*}

Since there is a $\mathbf{w} \in \fR^n$ with $\mathbf{w} \in E_{\alp(1)}(f_{a'})$ and $f_{a'}\{\mathbf{w}\}$ lopsided we have
\begin{eqnarray*}
	\mathbb{T}(f_{a'}) \ = \ \Big\{f' =  \kappa \cdot e^{i \cdot \psi_1} \cdot \mathbf{z}^{\alp(1)} + \sum_{j = 2}^d e^{i \cdot \psi_j} \cdot a_j \cdot \mathbf{z}^{\alp(j)} \ : \ \psi_j \in [0,2\pi) \text{ for all } j\Big\} & \subset & U_{\alp(1)}^A
\end{eqnarray*}
by Proposition \ref{PropLopsidedness}. Since furthermore $d \in \mathbb{T}(f_{a'})$, there exists a path $\gamma_3 \subset \mathbb{T}(f_{a'}) \subset U_{\alp(1)}^A$ from $a'$ to $d$.

Let $\gamma_4$ denote the line segment
\begin{eqnarray*}
	\gamma_4	& :=	& \big\{ d + \lam \cdot (0, \arg(b_2) \cdot (|b_2| - |a_2|) , \ldots , \arg(b_d) \cdot (|b_d| - |a_d|)), \quad \lam \in [0,1] \big\} \, .
\end{eqnarray*}
By construction $\gamma_4(\lam) \in \mathbb{T}(f_{a + \lam(b-a)})$ for all $\lam \in [0,1]$. Since for every $\lam \in [0,1]$ the first coordinate of $\gamma_4(\lam)$ has norm $\kappa$, it follows from \eqref{EquThmPathConstr1} and Proposition \ref{PropLopsidedness} that there is a $\mathbf{w} \in \fR^n$ such that $f_{\gamma_4(\lam)}\{\mathbf{w}\}$ is lopsided and in $E_{\alp(1)}(f_{\gamma_4(\lam)})$. Hence, $\gamma_4 \subset U_{\alp(1)}^A$. Therefore, $\gamma := \gamma_2 \circ \gamma_4 \circ \gamma_3 \circ \gamma_1$ is a path from $a$ to $b$ with $\gamma \in U_{\alp(1)}^A$.
\end{proof}

\begin{cor}
If $(\Delta,y)$ is barycentric then $U_y^{A}$ is path-connected.
\label{KorJackpot}
\end{cor}

\begin{proof}
All $S$ (see (\ref{EquSetS})) are simply connected (Lemma \ref{LemSimplyConnected}) and contain the origin (Theorem \ref{ThmRoughBounds}). Thus, $U_{y}^{A}$ is path-connected by Theorem \ref{ThmPathconstruction}.
\end{proof}

\section*{Acknowledgement}
The authors would like to thank Mikael Passare$^{\dag}$ (1959--2011) for helpful comments, in particular for his suggestion to investigate $A$-discriminants and the example presented in the beginning of Section \ref{SecBarycenter}.

Furthermore we thank an anonymous referee for numerous helpful suggestions and comments.

\bibliographystyle{amsplain}
\bibliography{bibamoebasgenus1}

\appendix
\section{Proof of Lemma~\ref{le:hypocycloid2}}

We provide the calculations for the proof of Lemma~\ref{le:hypocycloid2}.

\begin{lem}
Let $T$ denote the region whose boundary is the (rotated) hypocycloid given by $\phi \mapsto F(n,\phi)$ for $\phi \in [0,2\pi)$. Then $S \subseteq T$ and $\partial T \subseteq \partial S$.
\label{LemSsubsetT}
\end{lem}

\begin{proof}
By Theorem \ref{ThmMain2},
$S$ is given by the image of the function $g: [0,2\pi)^n \ra \fC, \phi \mapsto -|\Theta| \cdot \sum_{j = 0}^n e^{i \cdot (\arg(b_j) + \langle \alp(j), \phi \rangle)}$ 
with $\alp(0) = -\sum_{j = 1}^n \alp(j)$. 
In order to show $S \subseteq T$, it suffices to show
that every critical point of $g$ is contained in $T$,
because every boundary point of $S$ is a critical point of $g$.

Once more, we use the dual basis $\alp(1)^*,\ldots,\alp(n)^*$ of $\alp(1),\ldots,\alp(n)$) again, i.e. $\phi := \sum_{j = 1}^n \phi_j \cdot \alp(j)^*$. Furthermore, we can assume 
$\arg(b_1) = \cdots = \arg(b_n) = 0$ since we can replace $\phi_j$ by $-\arg(b_j) + \phi_j$. We have
\begin{eqnarray*}
    \frac{\partial g}{\partial \phi_j}(\phi)	& =	& -|\Theta| \cdot i \cdot \lf(e^{i \cdot \phi_j} - e^{-i \cdot (\sum_{l = 1}^n \phi_l - \arg(b_0))}\ri),
\end{eqnarray*}
and thus,
\begin{eqnarray*}
\RE\lf(\frac{\partial g}{\partial \phi_j}(\phi)\ri) & = & |\Theta| \cdot \lf(\sin(\phi_j) - \sin\lf(- \sum_{l = 1}^n \phi_l + \arg(b_0)\ri)\ri) \\
  \text{ and } \IM\lf(\frac{\partial g}{\partial \phi_j}(\phi)\ri) & = & |\Theta| \cdot \lf(-\cos(\phi_j) + \cos\lf(- \sum_{l = 1}^n \phi_l + \arg(b_0)\ri)\ri).
\end{eqnarray*}

$\phi$ is a critical point of $g$ if and only if $\RE(\nabla g(\phi)) = \lam_\phi \cdot \IM(\nabla g(\phi))$ with $\lam_\phi \in \fR$, i.e., if and only if for all $j \in \{1,\ldots,n\}:$
\begin{eqnarray*}
	  \lam_\phi \cdot \cos(\phi_j)  + \sin(\phi_j)	& =	& \lam_\phi \cdot \cos\lf(- \sum_{l = 1}^n \phi_l + \arg(b_0)\ri) + \sin\lf(- \sum_{l = 1}^n \phi_l + \arg(b_0)\ri).
\end{eqnarray*}
Since the right hand term is independent of $j$, this implies
\begin{eqnarray*}
	\lam_\phi \cdot \cos(\phi_j) + \sin(\phi_j)	& =	& \lam_\phi \cdot \cos(\phi_k) + \sin(\phi_k)
\end{eqnarray*}
for all $j,k \in \{1,\ldots,n\}$. This is in particular true if $\cos(\phi_j) = \cos(\phi_k)$ and $\sin(\phi_j) = \sin(\phi_k)$, i.e., if all $e^{i \cdot \phi_j}$ have the same argument, that is, $g(\phi)$ is located on the (rotated) hypocycloid given by $F(n,\psi), \psi \in [0,2\pi)$ (see \eqref{EquSimpleFiberDescription}, Lemma \ref{LemConstrF}).

The function $h(\phi_j) := \lam_\phi \cdot \cos(\phi_j) + \sin(\phi_j)$ 
is a periodic function in the interval $[0,2\pi)$ which has a vanishing derivative exactly
at the points $\phi_j$ with $\tan(\phi_j) = 1 / \lam_\phi$.
$\tan$ is $\pi$-periodic and strictly increasing on the interval $(-\pi/2,\pi/2)$.
Therefore, for a fixed solution $\phi_n$ of $h(\phi_n)$, for every $j \in \{1, \ldots, n-1\}$
there are exactly two possibilities: either $\phi_j = \phi_n$ or $\phi_j$ is the
unique solution distinct from $\phi_n$ with $h(\phi_j) = h(\phi_n)$, and that one
coincides with $- \sum_{l = 1}^n \phi_l + \arg(b_0)$.

Thus, if $\phi$ is a critical point with $g(\phi) \notin F(n,[0,2\pi))$, then there are $\phi_j$ (we choose here $j = 1,\ldots,s$ for some $1 \le s < n-1$ since every outer monomial has the same properties) satisfying 
$\lam_\phi \cdot \cos(\phi_j) = \cos\lf(- \sum_{l = 1}^n \phi_l + \arg(b_0)\ri)$ 
and $\sin(\phi_j) = \sin\lf(- \sum_{l = 1}^n \phi_l + \arg(b_0)\ri)$, 
which means that $\arg(e^{i \cdot \phi_1}) = \cdots = \arg(e^{i \cdot \phi_s}) = \arg(e^{- i \cdot \sum_{l = 1}^n \phi_l + \arg(b_0)})$ and 
$\arg(e^{i \cdot \phi_{s+1}}) = \cdots = \arg(e^{i \cdot \phi_n})$. 
Hence, $\phi_1 = \cdots = \phi_s$, $\phi_{s+1} = \cdots = \phi_n$ and
\begin{eqnarray*}
	\phi_1	& =	& - \sum_{l = 1}^n \phi_l + \arg(b_0) 
		\ =	\ - s \cdot \phi_1 - (n-s) \cdot \phi_n + \arg(b_0) \\
		& =	& - \frac{n-s}{s+1} \cdot \phi_n + \arg(b_0).
\end{eqnarray*}
Thus, $g(\phi)$ is located on the curve given by the hypocycloid with parameters 
$R = (n+1) |\Theta|$ and $r' = (s + 1) |\Theta|$ 
rotated by $\arg(b_0)$ (see \eqref{EquHypocycloidDef}).

Since $S$ is a subset of the closed ball $\cB_{(n+1) \cdot |\Theta|}(0)$ with radius $(n+1) \cdot |\Theta|$ around the origin (Theorem \ref{ThmRoughBounds}), it is bounded and since $U^A_y$ is a closed set, we have $\partial S \subset S$. 
Since the trajectory of every hypocycloid with parameters $R = n+1$ and $r \in \{2,\ldots,n-1\}$ is a subset of $T$ (coinciding with $F(n,\psi), \psi \in [0,2\pi)$ at the cusps), we have $S \subseteq T$.

Since the sets $S$ and $T$ are closed, the statement 
$\partial T \subseteq \partial S$ follows from $S \subseteq T$ and
$\partial T \subseteq S$. The first of these conditions has just
been shown and the second one is Lemma~\ref{LemConstrF} in connection
with the definition of $T$.
\end{proof}

\medskip

\noindent
\emph{Proof of Lemma~\ref{le:hypocycloid2}.}
By Lemma \ref{LemSsubsetT} we know that $S \subseteq T$ with $\partial T \subseteq \partial S$. 
Furthermore, by~\eqref{EquSimpleFiberDescription}
the image of $F$ is contained in $S$. Hence, the 
lemma is proven if we can show that the image of $F$ equals $T$ (which is simply connected by definition).

Let $k := n - 1 + (-1)^n$. We may assume $\arg(b_1),\ldots,\arg(b_n) = 0$ again (otherwise we transform the basis of $\phi_1,\ldots,\phi_n$ as in other proofs before). $F$ satisfies an 
$(n+1)$-quasiperiodicity condition $F(\mu,j \cdot \psi) = e^{i \cdot (2\pi j) / (n+1)} \cdot F(\mu,\psi)$ 
with $\mu \in [k,n], \psi \in [0,2\pi/(n+1)], j \in \{0,\ldots,n\}$. In particular,
\begin{eqnarray}
	\lf\{F(\mu,j \cdot 2\pi/(n+1)) : \mu \in [k,n]\ri\}	& =	& \lf\{e^{i \cdot 2\pi \cdot j / (n+1)} \cdot \mu : \mu \in [-k,n]\ri\} \label{EquLemSimplyConnectedProof1}
\end{eqnarray}
for $j \in \{0,\ldots,n\}$. 

We know that the path $\gamma_n(\psi) := F(n,\psi)$ with $\psi \in [0,2\pi)$ is a hypocycloid (Lemma \ref{LemConstrF}). Let $T = T_1 \cup \cdots \cup T_{n+1}$ where 
\begin{eqnarray*}
	  T_j & :=	& T \cap \{x \in \fC : \arg(x) \in [(j-1) \cdot 2\pi / (n+1), j \cdot 2\pi / (n+1)]\}.
\end{eqnarray*}
We show that the image of $F$ equals $T$ and thus is in particular simply connected. This follows 
from the quasiperiodicity, if the image of $F(\mu,\psi)$ with $\psi \in [0,2\pi/(n+1)]$ covers $T_1$.

The path-segment $\gamma_n(\psi)$ with $\psi \in [0,2\pi / (n+1)]$ is a loop-free path,
which is injective in the argument. The path-segment $\gamma_0(\psi)$ with $\psi \in [0,2\pi / (n+1)]$ is a segment of a circle in $(T \bs T_1) \cup \partial T_1$, which is also injective in the argument. Thus, for every $\psi \in (0,2\pi / (n+1))$ the segment $\sig_\psi := [F(0,\psi),F(n,\psi)]$ intersects $\{x \in [0,n)\} \cup \{x \in e^{i \cdot 2\pi / (n+1)} \cdot [0,n)\}$ at some point $t_{\psi}$. This implies with \eqref{EquLemSimplyConnectedProof1} that $F$ covers the homotopy $H: [0,2\pi/(n+1)] \ra \{[x_1,x_2] \subset \fC\}, \psi \ra [t_\psi,\gamma_n(\psi)]$ of line segments with $H(0) = [n,n], H(2\pi/(n+1)) = e^{i \cdot 2\pi / (n+1)} \cdot [n,n]$. The image of $H$ is $T_1$. Hence, the image of $F(\mu,\psi)$ with $\psi \in [0,2\pi/(n+1)]$ covers $T_1$ and therefore the image of $F$ is $T$. Since $\im(F) \subseteq S$ and $S \subseteq T$ we have $S = T$ and thus, $S$ is simply connected.
\ifpictures
\begin{figure}[ht]
	\begin{center}
	\begin{picture}(220,200)(0,0)
		\put(0,-20){\includegraphics[width=220pt]{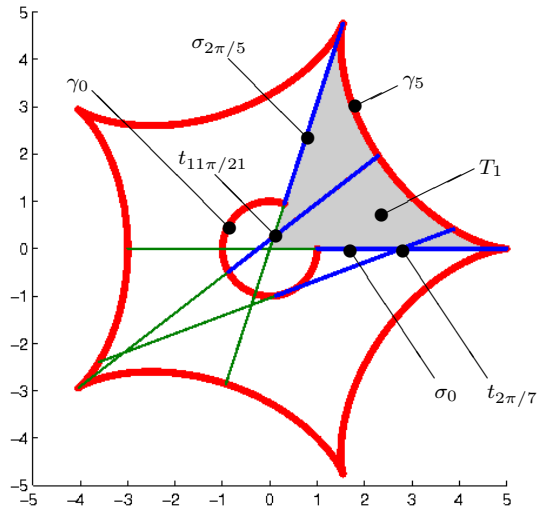}}

		\put(143,93.25){\circle*{5}}
		\put(143,93.25){\line(2,-3){32}}
		\put(175,37.5){\tiny{\mbox{$\sig_{0}$}}}
		
		\put(163,93.25){\circle*{5}}
		\put(163,93.25){\line(2,-3){32}}
		\put(195,37.5){\mbox{\tiny{$t_{2 \pi/7}$}}}
		
		\put(115,99){\circle*{5}}
		\put(115,99){\line(-1,1){23}}
		\put(78,126){\tiny{\mbox{$t_{11 \pi/21}$}}}
		
		\put(127,136){\circle*{5}}
		\put(127,136){\line(-1,1){30}}
		\put(83,170){\tiny{\mbox{$\sig_{2 \pi/5}$}}}
		
		\put(97.5,102){\circle*{5}}
		\put(97.5,102){\line(-1,1){52}}
		\put(36,155){\tiny{\mbox{$\gamma_0$}}}
		
		\put(155,107){\circle*{5}}
		\put(155,107){\line(2,1){35}}
		\put(192,122.5){\tiny{\mbox{$T_1$}}}
		
		\put(145,148){\circle*{5}}
		\put(145,148){\line(2,1){16}}
		\put(163,156){\tiny{\mbox{$\gamma_5$}}}
	\end{picture}
	\caption{Illustration of the covering of the set $S$ by the function $F$.}
	\label{FigureHomotopy}
	\end{center}
\end{figure}
\fi
\hfill $\Box$

\begin{exa}Figure \ref{FigureHomotopy} illustrates the proof of Lemma \ref{LemSimplyConnected} for the case of $f := c + \sum_{j = 1}^5 \mathbf{z}^{\alp(j)} \in \fC[z_1,\ldots,z_4]$ with $\sum_{j = 1}^5 \alp(j) = 0$. Here, $\Theta = 1$ hence $R = 5$ and $r = 1$. Due to quasiperiodicity it suffices to cover the grey region $T_1$. Of course, $\gamma_5$ is the hypocycloid with the upper values of $R$ and $r$ and $\gamma_0$ is the circle of radius $1$ around the origin. In the figure on can see the path-segments $\sig(0)$ and $\sig(\frac{2}{5} \pi)$ yielding the start- and endpoint of the homotopy $H$ (the two cusps intersecting $T_1$) and the path-segments $\sig(\frac{2}{7} \pi)$ and $\sig(\frac{11}{21} \pi)$ which yield $H(\frac{2}{7} \pi)$ and $H(\frac{11}{21} \pi)$ given by the subsegments from the point on $\gamma_5$ to $t_{2 \pi /7}$ resp. $t_{11 \pi /21}$. One can see how the complete area $T_1$ is covered by these subsegments given by $H$.
\end{exa}

\end{document}